\documentclass[11pt, DIV10,a4paper]{article}
\usepackage{color}
\usepackage{float}
\usepackage[bf,small]{caption2}
\usepackage{comment}
\usepackage{amsmath}
\usepackage{bigints}
\usepackage{rotating, booktabs}
\usepackage{amsopn}
\usepackage{amsfonts}
\usepackage{amsthm}
\usepackage{amssymb}
\usepackage{amsbsy}
\usepackage{multirow}
\usepackage{slashbox}
\usepackage{natbib}
\usepackage{tocbibind}
\usepackage{xr-hyper}
\usepackage[a4paper,colorlinks,breaklinks,bookmarksopen,bookmarksnumbered]{hyperref}
\usepackage[refpage]{nomencl}
\usepackage{makeidx}
\usepackage{pst-all}
\usepackage{epsfig,psfrag}
\usepackage{graphicx}
\usepackage{amsmath}
\usepackage{epstopdf}
\usepackage{tabularx}
\usepackage{subfigure}
\usepackage{setspace}
\usepackage[mathscr]{euscript}
\usepackage[margin=1in]{geometry}

\usepackage{amsfonts} 
\usepackage{calc}

\newcommand{\ueq}[1][]{%
  \if\relax\detokenize{#1}\relax
    \sbox0{$\underbrace{=}_{}$}%
    \mathrel{\mathmakebox[\wd0]{=}}
  \else
    \mathrel{\underbrace{=}_{\mathclap{#1}}}
  \fi}

\newcommand {\ctn}{\cite}

\usepackage{url}
\newcommand{\bx}{\mathbf x}
\newcommand{\bX}{\mathbf X}

\newcommand{\bY}{\mathbf Y}

\newtheorem{theorem}{Theorem}

\numberwithin{equation}{section}
\numberwithin{algo}{section}
\numberwithin{table}{section}
\numberwithin{figure}{section}

\numberwithin{equation}{section}
\theoremstyle{plain}
\newtheorem{as}{Assumption}
\newtheorem{lem}{Lemma}
\newtheorem{rmk}{Remark}

\normalsize

\begin{document}

\title{\textbf{ Posterior Convergence of Nonparametric Binary and Poisson Regression Under Possible Misspecifications}}
\author{Debashis Chatterjee$^{\dag}$ and Sourabh Bhattacharya$^{\dag, +}$ }
\date{}
\maketitle
\begin{center}
$^{\ddag}$ Indian Statistical Institute\\
$+$ Corresponding author:  \href{mailto: bhsourabh@gmail.com}{bhsourabh@gmail.com}
\end{center}

\begin{abstract}
	In this article, we investigate posterior convergence of nonparametric binary and Poisson regression 
	under possible model misspecification, assuming general stochastic process prior with appropriate properties. 
	Our model setup and objective for binary regression is similar to that of \ctn{Ghosal06} where the authors have used the 
	approach of entropy bound and exponentially consistent tests with the sieve method to achieve consistency with respect to their Gaussian process prior. 
	In contrast, for both binary and Poisson regression, using general stochastic process prior, our approach involves verification 
	of asymptotic equipartition property along with the method of sieve, which is a manoeuvre of the general results of \ctn{Shalizi09}, useful even 
	for misspecified models. Moreover, we will establish not only posterior consistency but also the rates at which the posterior probabilities converge, 
	which turns out to be the Kullback-Leibler divergence rate. We also investgate the traditional posterior convergence rates.
	Interestingly, from subjective Bayesian viewpoint we will show that the posterior predictive distribution can accurately approximate
	the best possible predictive distribution in the sense that the Hellinger distance, as well
	as the total variation distance between the two distributions can tend to zero, in spite of
	misspecifications.
	\\[2mm]
\end{abstract}

\section{Introduction}
The situation for applicability of nonparametric regression is frequently encountered in many practical scenarios where no parametric model fits the data. 
In particular, non-parametric regression for binary dependent variables is very common for various branches of statistics like medical and spatial statistics, whereas 
nonparametric version of Poisson regression is being used recently in many 
non-trivial scenerios such as for analyzing the likelihood and severity of vehicle crashes (\ctn{Ye18}). 
Interestingly, despite vast applicability of both the binary as well as Poisson regression, it seems that the available literature on 
nonparametric Poisson regression is scarce in comparison to the available literature on nonparametric binary regression. The Bayesian approach to nonparametric 
binary regression problem has been accounted for in \ctn{Diaconis93}. An account of posterior consistency for Gaussian process prior in nonparametric binary regression 
modeling can be found in \ctn{Ghosal06}, where the authors suggested that similar consistency results should hold for nonparametric Poisson regression model setup. 
Literature on consistency results for nonparametric Poisson regression is very limited. \ctn{Pillai07} have obtained consistency results for 
Poisson regression using an approach similar to that of \ctn{Ghosal06} under certain assumptions, but so far without explicit specifications and detail on prior. 
On the other hand, our approach will be based on results on \ctn{Shalizi09}, which is much different from \ctn{Ghosal06} and capable of handling model misspecification.
Unlike the previous works, the approach of \ctn{Shalizi09} also enables us to investigate the rate at which the posterior converges, which turns out to be the
Kullback-Leibler (KL) divergence rate, and also the traditional posterior convergence rate.

In this article, we investigate posterior convergence of nonparametric binary and Poisson regression 
where the nonparametric regression is modeled as some suitable stochastic process. In the binary situation, we consider
a similar setup as that of \ctn{Ghosal06}, where the authors have considered binary observations 
with response probability as an unknown
smooth function of a set of covariates, which was modeled using Gaussian process. Here we will consider a binary response variable $Y$ and a $d$-dimensional 
covariate $x$ belonging to a compact subset. The probability function is given by $p(x)=P(Y=1|X=x)$ along with a prior for $p$ induced by some appropriate stochastic 
process $\eta(x)$ with the relation $p(x)=H\left(\eta(x)\right)$ for a known, non-decreasing and continuously differentiable cumulative distribution function $H(\cdot)$. 
We will establish a posterior convergence theory for nonparametric binary regression under possible misspecifications 
based on the general theory of posterior convergence of \ctn{Shalizi09}. Our theory also includes the case of misspecified models,
that is, if the true regression function is not even supported by the
prior. This approach to Bayesian asymptotics also permits us to show that the relevant posterior probabilities converge
at the KL divergence rate, and that the posterior convergence rate with
respect to KL-divergence is just slower than $\frac{1}{n}$, where $n$ denotes the number of observations. We
further show that even in the case of misspecification, the posterior predictive distribution can
approximate the best possible predictive distribution adequately, in the sense that the Hellinger
distance, as well as the total variation distance between the two distributions can tend to zero. 

For nonparametric Poisson regression, given $x$ in the compact space of covariates, we model the mean function $\lambda(x)$ as $\lambda(x)=H(\eta(x))$, 
where $H$ is a continuously differentiable function. Again, we investigate the general theory of posterior convergence, including misspecifications,
rate of convergence of the posterior distribution and the usual posterior convergence rate, in Shalizi's framework.

The rest of our paper is structured as follows.
In Section \ref{ss1} we provide a brief overview and intuitive explanation of the main assumptions and results of \ctn{Shalizi09} suitable for our approach.
The basic prenises for nonparametric binary and Poisson regression are provided in Sections \ref{binary_model} and \ref{poisson_model}, respectively.
The required assumptions and their discussions are provided in Section \ref{ssAssumption}.
In Section \ref{convergence}, our main results on posterior convergence of binary and Poisson regression are provided, while Section \ref{miss} details the consequences
of misspecifications. Concluding remarks are provided in Section \ref{conclusion}.

The technical details are presented in the Appendix. 
Specifically, details of the necessary assumptions and results of \ctn{Shalizi09} are provided in Appendix \ref{ASZ}. The detailed proofs of verification of Shalizi's
assumptions are provided in Appendix \ref{Shaverify} and Appendix \ref{PShaverify} for binary and Poisson regression setups, respectively.


  \section{ An outline of the main assumptions and results of Shalizi}
\label{ss1}

Let the set of random variables for the response be denoted by $\mathbf{Y_n} = \left(Y_1, Y_2, \ldots, Y_n \right)$. For a given parameter space $\Theta$, let $f_{\theta}(\mathbf{Y_n})$ be the observed likelihood and $f_{\theta_0}(\mathbf{Y_n})$ be the true likelihood. We assume $\theta \in \Theta$ but the truth $\theta_0$ need not be in $\Theta$, 
thus allowing possible misspecification.

The KL divergence $KL(f, g)=\int f \log (\frac{f}{g})$ is a measure of divergence between two probability densities $f$ and $g$.
The KL divergence is related to likelihood ratios, since by the Strong Law of Large Numbers (SLLN) for independent and identical ($iid$) situations, 
\begin{equation*}
\frac{1}{n}\displaystyle \sum_{i=1}^{n} \log\left[\dfrac{f(Y_i)}{g(Y_i)}\right] \rightarrow KL(f, g ). 
\end{equation*}
For every $\theta\in\Theta$, the KL divergence rate is given by:
\begin{equation}
h(\theta)=\underset{n\rightarrow\infty}{\lim}~\dfrac{1}{n}E\left[\log\left\{\dfrac{f_{\theta_0}(\mathbf{Y_n})}{f_{\theta}(\mathbf{Y_n})}\right\}\right].
\label{eq:A2}
\end{equation}

The key ingredient associated with
the approach of \ctn{Shalizi09} for proving convergence of the posterior distribution of $\theta$ is to show that the
asymptotic equipartition property holds. 
To illustrate, let us consider the following likelihood
ratio:

\begin{equation}
R_n(\theta)=\dfrac{f_{\theta}(\mathbf{Y_n})}{f_{\theta_0}(\mathbf{Y_n})}.
\end{equation}

If we think of the $iid$ setup, $h(\theta)$ reduces to the KL divergence between the true and the
hypothesized model.
For each $\theta\in\Theta$, ``asymptotic equipartition"  property is as follows:

\begin{equation}
\underset{n\rightarrow\infty}{\lim}~\frac{1}{n}\log \left[R_n(\theta)\right]=-h(\theta),
\label{eq:e1}
\end{equation}
Here
``asymptotic equipartition" refers to dividing up $\log \left[R_n(\theta)\right]$ into $n$ factors for large $n$ 
such that all the factors are asymptotically equal. For illustration, in the $iid$ scenario, each factor converges to the same KL divergence between the true and the postulated model. 
The purpose of asymptotic equipartition is to ensure that relative to the true distribution, 
the likelihood of each $\theta$ decreases to zero exponentially
fast, with rate being the KL divergence rate.

for $A\subseteq\Theta$, let
\begin{align}
h\left(A\right)&=\underset{\theta\in A}{\mbox{ess~inf}}~h(\theta);\label{eq:h2}\\
J(\theta)&=h(\theta)-h(\Theta);\label{eq:J}\\
J(A)&=\underset{\theta\in A}{\mbox{ess~inf}}~J(\theta),\label{eq:J2}
\end{align}
where  $ h(A)$ roughly represent the minimum KL-divergence between the postulated
and the true model over the set A. If $h(\Theta) > 0$, it indicates model misspecification. However, as we shall show, 
model misspecification need not always imply that $h(\Theta) > 0$. One such counter example is also given in \ctn{Chatterjee18a}.

Observe that, for $A \subset \Theta$, $J(A)>0$.
For the prior, it is required to construct an appropriate sequence of sieve sets  $\mathcal G_n\rightarrow\Theta$ as $n\rightarrow\infty$ 
such that: 

\begin{enumerate}
	\item[(1)] $h\left(\mathcal G_n\right)\rightarrow h\left(\Theta\right)$, as $n\rightarrow\infty$.
	\item[(2)]
	$\pi\left(\mathcal G_n\right)\geq 1-\alpha\exp\left(-\beta n\right),~\mbox{for some}~\alpha>0,~\beta>2h(\Theta)$;
	
\end{enumerate}

The sets $\mathcal G_n$ can be  interpreted as the sieves 
in the sense that, the behaviour of the likelihood ratio and the posterior on the sets $\mathcal G_n$ essentially carries over to $\Theta$. 

Let $\pi(\cdot|\mathbf{Y}_n)$ denote the posterior distribution of $\theta$ given $\mathbf{Y}_n$. Then with the above notions, verification of \eqref{eq:e1} along with several other technical conditions (details given in Appendix \ref{ASZ}) ensure that  any $A \subseteq \Theta$ for which $\pi(A)>0$,   
\begin{equation}
	\underset{n\rightarrow\infty}{\lim}~\pi(A|\bY_n)=0,
	\label{eq:post_conv1}
\end{equation}
almost surely, provided that $h(A)>h(\Theta)$. 
The latter $h(A)>h(\Theta)$ implies positive KL-divergence in $A$, even if $h(\Theta)=0$. That is, $A$ is the set in which the postulated model
fails to capture the true model in terms of the KL-divergence. Hence, expectedly, the posterior probability of that set converges to zero. 

Under mild assumptions, it also holds that
\begin{equation}
	\underset{n\rightarrow\infty}{\lim}~\frac{1}{n}\log\pi(A|\bY_n)=-J(A),
	\label{eq:post_conv2}
\end{equation}
almost surely.
This result shows that the rate at which the posterior probability of $A$ converges to zero is about $\exp(-nJ(A))$. From the above results it is clear 
that the posterior concentrates
on sets of the form $N_{\epsilon}=\left\{\theta:h(\theta)\leq h(\Theta)+\epsilon\right\}$, for any $\epsilon>0$.

Shalizi addressed the rate of posterior convergence as follows. Letting 
$N_{\epsilon_n}=\left\{\theta:h(\theta)\leq h(\Theta)+\epsilon_n\right\}$, where $\epsilon_n\rightarrow 0$ such that $n\epsilon_n\rightarrow\infty$,
Shalizi showed, under an additional technical assumption, that almost surely, 
\begin{equation}
	\underset{n\rightarrow\infty}{\lim}~\pi\left(N_{\epsilon_n}|\bY_n\right)=1.
	\label{eq:conv_rate}
\end{equation}

Moreover, it was shown by Shalizi that the squares of the Hellinger and the total variation
distances between the posterior predictive distribution and the best possible predictive distribution under the truth, 
are asymptotically almost surely bounded above by $h(\Theta)$ and $4h(\Theta)$,
respectively. That is, if $h(\Theta)=0$, then this allows very accurate approximation of the true predictive distribution by the posterior predictive distribution. 


\section{Model setup and preliminaries of the binary regression }
\label{binary_model}
Let $Y \in \{0, 1\}$ be a binary
outcome variable and X a vector of covariates. 
Suppose $Y_1, Y_2, \ldots, Y_n \in  \{0, 1\}^n$ are some independent binary responses conditional on unobserved  covariates $X_1, X_2, \ldots, X_n \in \mathfrak{X} \subset \Re^d$.  We assume that the covariate space  $\mathfrak{X}$ is compact. Let $\mathbf{Y}_n =(Y_1, Y_2, \ldots, Y_n)^T $ be the binary response random variables against the covariate vector $\mathbf{X}_n=(X_1, X_2, \ldots, X_n)^T$. The corresponding observed values will be denoted by $\mathbf{y}_n=(y_1, y_2, \ldots, y_n)$ and $\mathbf{x}_n=(x_1, x_2, \ldots, x_n)$ respectively. 
Let the model be specified as follows: for $ i=1, 2, \ldots, n$:
\begin{eqnarray}\label{e1}
& Y_i| X_i \displaystyle \sim Binomial\left( 1, p(X_i)\right) \\
& p(x) = H \left(\eta(x)\right) \\
	& \eta(\cdot) \sim \pi_{\eta}, 
\end{eqnarray}
where $\pi_{\eta}$ is the prior for some suitable stochastic process.
Note that the prior for $p$ is induced by the prior for $\eta$.
Our concern is to infer about  the success probability function  $p(x) =P(Y = 1|X=x)$  when the number of observations goes to infinity. 
We will assume that the functions $\eta$ have continuous first partial derivatives. We denote this class of functions by $\mathcal{C'}(\mathfrak{X})$. 
We do not assume the truth $\eta_0$ in $ \mathcal{C'}(\mathfrak{X})$, allowing misspecification.
The link function $H$ is a known, non-decreasing, continuously differentiable cumulative distribution function on the real line $\Re$. 
It is widely accepted to assume the function $H(\cdot)$ to be known as part of model assumption. For example, in  logistic regression we choose the standard 
logistic cumulative distribution function as the link function, whereas in probit regression $H$ is chosen to be the standard normal cumulative distribution function $\phi$. 
More  discussion on link function along with several other examples can be found in \ctn{Choudhuri07}, \ctn{Newton96}, \ctn{Gelfand91}.  A Bayesian method for estimation of $p$ 
has been provided in \ctn{Choudhuri07}. In has been shown in \ctn{Ghosal06} that the sample paths of the Gaussian processes can well approximate a large class of functions 
and hence it is not essential to consider additional uncertainty in the link function $H$. 

Let $\mathfrak{C}$ be the counting measure on $\{0,1\}$.  
Then according to the model assumption, the conditional density of $y$ given $x$ with respect to $\mathfrak{C}$ will be represented by the 
density function $f$ as follows:

\begin{equation}
f(y|x)= p(x)^y \left(1-p(x)\right)^{1-y}.\label{eq:binary_density}
\end{equation} 
The prior for $f$ will be denoted by $\pi$.
Let $f_0$ and $p_0$ denote truth density and success probability, respectively. Then under the truth, the joint density is:

\begin{equation}
f_{0}(y|x)= p_{0}(x)^y \left(1-p_0(x)\right)^{1-y}.
\label{eq:binary_density_true}
\end{equation}

One of the main objectives of this article is to show consistency of the posterior distribution of $p$ treated as parameter arising 
from the parameter space $\Theta$ specified as follows: 
\begin{equation} \label{Theta}
\Theta =\left \lbrace  p(\cdot): p(x)=H\left( \eta(x)\right), \eta \in \mathcal{C'}(\mathfrak{X}) \right \rbrace,
\end{equation} 
or simply, $\Theta=\mathcal{C'}(\mathfrak{X})$.



\section{Model setup and preliminaries of Poisson regression }
\label{poisson_model}
For Poisson regression model set up, let $Y \in \mathbb{N}$ be a count 
outcome variable and $X$ a vector of covariates. Here $\mathbb{N}$ denote the set of non negative integers. 
Suppose $Y_1, Y_2, \ldots, Y_n \in \mathbb{N}^n$ are some independent responses conditional on covariates $X_1, X_2, \ldots, X_n \in \mathfrak{X} \subset \Re^d$.  
We assume that the covariate space  $\mathfrak{X}$ is compact. Let $\mathbf{Y}_n =(Y_1, Y_2, \ldots, Y_n)^T $ be the response random variables against the covariate vector $\mathbf{X}_n=(X_1, X_2, \ldots, X_n)^T$. The corresponding observed values will be denoted by $\mathbf{y}_n=(y_1, y_2, \ldots, y_n)$ and $\mathbf{x}_n=(x_1, x_2, \ldots, x_n)$ respectively. 
Let the parameter space be specified as follows:
\begin{equation} \label{Lambda}
\Lambda =\left \lbrace  \lambda(\cdot): \lambda(x)=H\left( \eta(x)\right), \eta \in \mathcal{C'}(\mathfrak{X}) \right \rbrace.\end{equation} 
The link function $H$ is a known, non-negative continuously differentiable function on $\Re$. We equivalently define the parameter space as $\Theta=\mathcal{C'}(\mathfrak{X})$.
Thus, in what follows, we shall use both $\Lambda$ and $\Theta$ to denote the parameter space, depending on convenience.
Then the model is specified as follows: for $ i=1, 2, \ldots, n$,
\begin{eqnarray}\label{Pe1}
& Y_i|X_i \sim  \exp\left(-\lambda(X_i)\right)\dfrac{(\lambda(X_i))^y}{y!} \\
& \lambda(x) = H \left(\eta(x)\right); \\
	& \eta(\cdot) \sim \pi_{\eta}. 
\end{eqnarray}
Similar to binary regression, here our concern will be to infer about $\lambda(x)$ when the number of observations goes to infinity. 
We do not assume the truth $\eta_0$ in $ \mathcal{C'}(\mathfrak{X})$ as before, allowing misspecification. 

Now, suppose $\mathfrak{C}$ be the counting measure on $\mathbb{N}$.  
According to the model assumption for Poisson regression, the conditional density of $y$ given $x$ with respect to 
$\mathfrak{C}$ will be represented by density function $f$ as follows:

\begin{equation}
f(y|x)= \exp\left(-\lambda(x)\right)\dfrac{(\lambda(x))^y}{y!}.\label{density}
\end{equation} 
The prior for $f$ will be denoted by $\Pi$.
Let $f_0$ and $\lambda_0$ denote truth density and true mean function, respectively. 
Again, one of our main aims is to establish consistency of the posterior distribution of $\lambda$ treated as parameter arising from $\Lambda$.



\section{ Assumptions and their discussions }
\label{ssAssumption}

We need to make some appropriate assumptions for establishing convergence of both the binary and Poisson regression models equipped with stochastic process prior. The
latter also requires suitable assumptions.
Many of the assumptions are similar to those taken in \ctn{Chatterjee18a}. Hence the purpose of such assumptions will be as discussed in 
\ctn{Chatterjee18a}, which we shall briefly 
touch upon here. 

\begin{as}\label{AA1}
	$\mathfrak{X}$ is a compact, $d$-dimensional space, for some finite $d \geq 1$ equipped with a suitable metric.
\end{as}

\begin{as}\label{AA2}
Recall that in our notation, $\mathcal{C}'(\mathfrak{X})$ denotes the class of continuously partially differentiable function on $\mathfrak{X}$. 
In other words, the functions $\eta\in \mathcal{C}'(\mathfrak{X})$ are continuous on $\mathfrak{X}$ and for such functions the limit
	\begin{equation}
	\eta_{j}'(\mathbf{x})=\dfrac{\partial \eta(\mathbf{x})}{\partial x_j}= \lim_{h \rightarrow 0} \dfrac{\partial \eta\left(\mathbf{x}+h \mathbf{\delta}_j\right)-\eta(\mathbf{x})}{\partial h}
	\end{equation}
	exists for each $\mathbf{x} \in \mathfrak{X}$ and is  continuous $\mathfrak{X}$. Here $\mathbf{\delta}_j$ is the
	d-dimensional vector with the $j$-th element as 1 and all the other elements as zero. 
\end{as}

	\begin{as}\label{AA3}
		

The priors for $\eta$ is chosen such that for $\beta>2h\left(\Theta\right)$, 
\begin{align}
	\pi\left(\|\eta\|\leq\exp\left(\left(\beta n\right)^{1/4}\right)\right)&\geq 1-c_{\eta}\exp\left(-\beta n\right);\notag\\
	\pi\left(\|\eta^\prime_j\|\leq\exp\left(\left(\beta n\right)^{1/4}\right)\right)&\geq 1-c_{\eta^\prime_j}\exp\left(-\beta n\right),~\mbox{for}~j=1,\ldots,d;\notag
\end{align}
where $c_{\eta}$ and $c_{\eta^\prime_j}$; $j=1,\ldots,d$, are positive constants.


\end{as}


We treat the covariates as either random (observed or unobserved) or non-random (observed). 
Accordingly, in Assumption \ref{AA4} below we provide conditions pertaining to these aspects.

\begin{as}\label{AA4}
	\begin{itemize}
		\item[(i)] $\{x_i: i = 1, 2, \ldots \}$ is an observed or unobserved sample associated with an iid sequence associated with some probability measure $Q$, supported on $\mathfrak{X}$, which is
		independent of $\{y_i: i = 1, 2, \ldots \}$
		\item[(ii)] $\{x_i: i = 1, 2, \ldots \}$ is an observed non-random sample. In this case, we consider a
		specific partition of the $d$-dimensional space $\mathfrak{X}$ into n subsets such that each subset
		of the partition contains at least one $x \in  \{x_i: i = 1, 2, \ldots \}$ and has Lebesgue measure $\frac{L}{n}$, for some $L > 0$.
	\end{itemize}
\end{as}

\begin{as}\label{AA5}
	The truth function $\eta_0$ is bounded in sup norm. In other words, the truth $\eta_0$ satisfies the following for some constant $\kappa_0$ :
	\begin{equation}
		\|\eta_0\|_{\infty} < \kappa_0 <\infty
	\end{equation} 
	Observe that in general $\eta_0 \notin \mathcal{C}'(\mathfrak{X})$. For random covariate $X$, we assume that $\eta_{0}(X)$ is measurable.
\end{as}


\begin{as}\label{AA6}
	For binary regression model set up we assume a uniform positive lower bound $\kappa_B$ for $\min\{ p(\cdot), 1-p(\cdot)\}$. In other words, for all $p\in \Theta$,
		\begin{equation}
\inf \{\min \left( p(x), 1-p(x) \right):  \ x \in \mathfrak{X}\}\geq \kappa_B >0, 
	\end{equation} 
where $\Theta$ as defined in expression \ref{Theta}.
\end{as}

\begin{as}\label{AA7}
For Poisson regression model set up we assume a uniform positive lower bound $\kappa_P$ for $\lambda(\cdot)$. In other words, for all $\lambda \in \Lambda$,
	\begin{equation}
	\inf \{\lambda(x) :  \ x \in \mathfrak{X}\}\geq \kappa_P >0, 
	\end{equation} 
	where $\Lambda$ is as defined in expression \ref{Lambda}.
\end{as}






\subsection{Discussion of the assumptions}
\label{subsec:ass_diss}

Assumption \ref{AA1} is on compactness of $\mathfrak{X}$, which  guarantees that continuous functions on $\mathfrak{X}$ will have finite sup-norms.

Assumption \ref{AA2} is as taken in \ctn{Chatterjee18a} for the purpose of  constructing appropriate sieves in order to show  posterior
convergence results. More precisely, Assumption \ref{AA2} is required for to ensure that $\eta$ is Lipschitz continuous in
the sieves. Since a differentiable function is Lipschitz if and only if its partial derivatives are bounded, this serves our purpose, as continuity of the partial derivatives 
of $\eta$ guarantees the boundedness in the
compact domain $\mathfrak{X}$. In particular, if $\eta$ is a Gaussian process, conditions presented in 
\ctn{Adler81}, \ctn{Adler07}, \ctn{Cramer67} guarantee the above continuity and smoothness properties
required by Assumption \ref{AA2}. We refer to \ctn{Chatterjee18a} for more discussion about this.

Assumption \ref{AA3} is required for ensuring that the complements of the sieves have exponentially small probabilities. In particular, this assumption is
satisfied if $\eta$ is a Gaussian process, even if $\exp\left(\left(\beta n\right)^{1/4}\right)$ is replaced with $\sqrt{\beta n}$.

Assumption \ref{AA4} is for the covariates $x_i$, accordingly as they are considered an observed random sample, unobserved random sample, or non-random. 
Note that thanks to the strong law of large numbers (SLLN), given any $\eta$ in the complement of some null set with respect to the prior, 
and given any sequence
$\left\{\bx_i:i=1,2,\ldots\right\}$ 
Assumption \ref{AA4} (i) ensures that for any integrable function $g$, 
as $n\rightarrow\infty$,
\begin{equation}
	\frac{1}{n}\sum_{i=1}^ng(\bx_i)\rightarrow\int_{\mathfrak X}g(\bx)dQ(\bX)
	=E_{\bX}\left[g(\bX)\right]~\mbox{(say)},
	\label{eq:a3_1}
\end{equation}
where $Q$ is some probability measure supported on $\mathfrak X$.

Assumption \ref{AA4} (ii) ensures that $\frac{1}{n}\sum_{i=1}^ng(\bx_i)$ is a particular Riemann sum and hence (\ref{eq:a3_1}) holds with
$Q$ being the Lebesgue measure on $\mathfrak X$. 
We continue to denote the limit in this case by $E_{\bX}\left[g(\bX)\right]$.


Assumption \ref{AA5} is equivalent to the  Assumption(T) of \ctn{Ghosal06}. Assumption \ref{AA5}  actually implies that $p_0(x) = H(\eta_{0}(x))$ is bounded away 
from 0 and 1 and hence the corresponding truth function $\eta_{0}$ given by $\eta_{0}(x)= H^{-1}(p_0(x))$ is uniformly bounded above and below. 


As $\eta_0$ is uniformly bounded above and below, hence $p_0(x)=H(\eta_{0}(x))$ will also be bounded away from 0 and 1. 
For the Poisson regression model set up it follows that $\|\lambda_0\|_{\infty}<\infty$. 

It is to be noted that here we do not require to assume that $ p_0 \in \Theta$ or $ \lambda_0 \in \Lambda$, allowing model misspecifications.
 
	Observe that, similar to 
	\ctn{Pillai07} we need the parameter space for Poisson regresion to be bounded away from
	zero (Assumption \ref{AA7}). As pointed out in \ctn{Pillai07}, we cannot bypass this and as such these are not a mere pathway towards our proof. 
This is because, if  almost all observations in a sample from a Poisson distribution are zero, then it impossible to extract the information about the (log) mean. 
Hence we must require at least some condition to make it bound away from zero. Similar argument also applicable for binary regression, which is reflected in
	Assumption \ref{AA6}. 

	It is important to remark that Assumptions \ref{AA6} and \ref{AA7} are necessary only to validate Assumption (S6) of Shalizi, and unnecessary elsewhere.
	The reasons are clarified in Remarks \ref{rmk:hoeff_binary} and \ref{rmk:hoeff_poisson}. Although many of our proofs would be simpler if 
	Assumptions \ref{AA6} and \ref{AA7} were used, we reserved these assumptions only to validate Assumption (S6) of Shalizi.

	To achieve Assumptions \ref{AA6} and \ref{AA7}, we set, for all $x\in\Re$, 
	\begin{equation}
		H(x)=\kappa_B\mathbb I_{\left\{G(x)\leq\kappa_B\right\}}(x)+G(x)\mathbb I_{\left\{\kappa_B<G(x)< 1-\kappa_B\right\}}(x)
		+(1-\kappa_B)\mathbb I_{\left\{G(x)\geq 1-\kappa_B\right\}}(x),
		\label{eq:binary_H}
	\end{equation}
	for the binary case, where $0<\kappa_B<1/2$, and
	\begin{equation}
		H(x)=\kappa_B\mathbb I_{\left\{G(x)\leq\kappa_P\right\}}(x)+G(x)\mathbb I_{\left\{G(x)> \kappa_P\right\}}(x),
		\label{eq:poisson_H}
	\end{equation}
	where $\kappa_P>0$.
	In (\ref{eq:binary_H}), $G$ is a continuously differentiable distribution function on $\Re$ and in (\ref{eq:poisson_H}), $G$ is a non-negative continuously
	differentiable function on $\Re$.

\section{Main results on posterior convergence }
\label{convergence}

Here we will state a summary of our main results regarding posterior convergence of nonparametric binary regression and Poisson regression. 
The key results associated with the asymptotic
equipartition property are provided in Theorems \ref{T1} -- \ref{PT2}, proofs of which 
are provided in 
Appendix \ref{Shaverify} (for binary regression) and in Appendix \ref{PShaverify} (for Poisson regression).

\begin{theorem} \label{T1}
	Let $Q$ and the counting measure $\mathfrak{C}$ on $\{0, 1\}$ be the measures associated with the random variable $X$ and the binary random variable $Y$ respectively. 
	Denote $E_{\mathbf{X}, \mathbf{Y}}(\cdot)= \int \int \cdot \ d\mathfrak{C} \ dQ$ and $E_{\mathbf{X}}(\cdot)= \int \int \cdot \   dQ$. Then under the 
	nonparametric binary regression model, under Assumption \ref{AA4}, the KL divergence rate $h(p)$ exists for $p \in \Theta$, and is given by
	\begin{equation}\label{h(p)}
	h(p)=\left[E_{\mathbf{X}} \left( p_{0}(\mathbf{X})\log \left\lbrace \dfrac{p_{0}(\mathbf{X})}{p(\mathbf{X})}\right\rbrace \right)+ E_{\mathbf{X}}\left((1-p_{0}(\mathbf{X})) \log \left\lbrace\dfrac{ \left(1- p_{0}(\mathbf{X})\right)}{\left(1- p(\mathbf{X})\right)} \right\rbrace\right) \right ].
	\end{equation}
	Alternatively, $h(p)$ admits the following form:
	\begin{equation}
	h(p)=E_{\mathbf{X}, \mathbf{Y}} \left( f_{0}(\mathbf{X}, \mathbf{Y})\log \left\lbrace \dfrac{f_{0}(\mathbf{X}, \mathbf{Y})}{f(\mathbf{X}, \mathbf{Y})}\right\rbrace \right), 
	\end{equation}
	where $f$ and $f_0$ are as defined in \eqref{eq:binary_density} and \eqref{eq:binary_density_true}.
\end{theorem}

\begin{theorem} \label{PT1}
	Let $Q$ and the counting measure $\mathfrak{C}$ on $\mathbb{N}$ be associated with the random variable $X$ and the count random variable $Y$, respectively. 
	Denote  $E_{\mathbf{X}, \mathbf{Y}}(\cdot)= \int \int \cdot \ d\mathfrak{C} \ dQ$ and $E_{\mathbf{X}}(\cdot)= \int \int \cdot \   dQ$. Then under the nonparametric
	Poisson regression model, under Assumption \ref{AA4}, the KL divergence rate $h(\lambda)$ exists for $\lambda \in \Lambda$, and is given by
	\begin{equation}\label{Ph(p)}
	h(\lambda)=\left[E_{\mathbf{X}} \left(    \lambda(\mathbf{X})-\lambda_{0}(\mathbf{X})  \right)+ E_{\mathbf{X}} \left( \lambda_{0}(\mathbf{X})\log \left\lbrace \dfrac{\lambda_{0}(\mathbf{X})}{\lambda(\mathbf{X})}\right\rbrace \right) \right ].
	\end{equation}
	
\end{theorem}

\begin{theorem} \label{T2}
	Under the nonparametric binary regression model and Assumption \ref{AA4}, the asymptotic equipartition property holds, and is given by 
	\begin{equation}
	\underset{n\rightarrow\infty}{\lim}~\frac{1}{n}\log \left[R_n(p)\right]=-h(p).
	\end{equation}
	The convergence is uniform on any compact subset of $\Theta$.
\end{theorem}

\begin{theorem} \label{PT2}
	Under the nonparametric Poisson regression model and Assumption \ref{AA4}, the asymptotic equipartition property holds, and is given by 
	\begin{equation}
	\underset{n\rightarrow\infty}{\lim}~\frac{1}{n}\log \left[R_n(\lambda)\right]=-h(\lambda).
	\end{equation}
	The convergence is uniform on any compact subset of $\Lambda$.
\end{theorem}
Theorems \ref{T1} and \ref{T2} for binary regression and Theorems \ref{PT1} and \ref{PT2} for Poisson regression ensure that conditions (S1) to (S3) of \ctn{Shalizi09} hold, and (S4) holds for both binary and Poisson regression because of compactness of $\mathfrak X$ and continuity of $H$ and $\eta$. 
The detailed proofs are presented in Appendix \ref{subsec:S4} and Appendix \ref{subsec:PS4}, respectively.  

We construct the sieves $\mathcal{G}_{n}$ for binary regression model set up as follows:

\begin{equation}\label{sieveset1}
\begin{aligned}
	\mathcal{G}_n={} & \{  \eta\in\mathcal C'(\mathfrak X): \|\eta\| \leq \exp(\left(\beta n\right)^{1/4}),  \\
	&  \|\eta_{j}'\| \leq \exp(\left(\beta n\right)^{1/4}); j=1, 2, \ldots, d\}
\end{aligned}
\end{equation}

It follows that $\mathcal{G}_{n} \rightarrow \Theta$ as $n \rightarrow \infty$, where the parameter space $\Theta$ is given by \eqref{Theta}.


In a similar manner, we construct the sieves $\mathbb{G}_{n}$ for binary regression as follows:

\begin{equation}\label{Psieveset1}
\begin{aligned}
	\mathbb{G}_n={} & \{  \lambda(\cdot) : \ \ \lambda(x)=H(\eta(x)), \eta\in\mathcal C'(\mathfrak X), \|\eta\| \leq \exp(\left(\beta n\right)^{1/4}),  \\
	&  \|\eta_{j}'\| \leq \exp(\left(\beta n\right)^{1/4}); j=1, 2, \ldots, d\}.
\end{aligned}
\end{equation}

Then similarly it will also follow that $\mathbb{G}_{n} \rightarrow \Lambda$ as $n \rightarrow \infty$, where the parameter space $\Lambda$ is given by \eqref{Lambda}.



Assumption \ref{AA3} ensures that for binary regression, 
$\Pi \left( \mathcal{G}^{c}_{n}\right)\leq \alpha \exp(- \beta n)$ for some $\alpha >0$ and similarly 
$\Pi \left(  \mathbb{G}^{c}_{n}\right)\leq \alpha \exp(- \beta n)$ for Poisson regression. Now, these results, continuity of $h(\theta)$, $h(\lambda)$
(the proofs of continuity of $h(p)$ and $h(\lambda)$ follows using the same techniques as in Appendices \ref{subsec:S1} and \ref{subsec:PS1}),
compactness of $\mathcal{G}_n$, $\mathbb{G}_n$ and the uniform convergence results of Theorems \ref{T2} and \ref{PT2}, together ensure (S5) for both the model setups.

Now, as pointed out in \ctn{Chatterjee18a}, we observe that the aim of assumption (S6) is to ensure that (see the proof of Lemma 7
of \ctn{Shalizi09}) for every $\epsilon > 0$ and for all  sufficiently large $n$,

\begin{equation}
\dfrac{1}{n} \log \int_{\mathcal{G}_{n} } R_{n}(p) \ d\pi(p) \leq h(\mathcal{G}_{n}) +\epsilon, \ \ \text{almost surely}.
\end{equation}

As $h(\mathcal{G}_{n}) \rightarrow h(\Theta)$ as $n \rightarrow \infty $, it is enough to verify that for every $\epsilon >0$ and for all $n$ 
sufficiently large,

\begin{equation}
\dfrac{1}{n} \log \int_{\mathcal{G}_{n} } R_{n}(p) \ d\pi(p) \leq h(\Theta) +\epsilon, \ \ \text{almost surely}.
\end{equation}

First we observe that 
\begin{equation}
\dfrac{1}{n} \log \int_{\mathcal{G}_{n} } R_{n}(p) \ d\pi(p) \leq  \dfrac{1}{n} \sup_{p \in \mathcal{G}_{n}} \log R_{n}(p).
\end{equation}
For large enough  $\kappa> h(\Theta)$,  consider  $ S = \{p : h(p) \leq \kappa \}$.


\begin{lem}\label{Scompact}
	$ S = \{p : h(p) \leq \kappa \}$ is a compact set.
\end{lem}
\begin{proof}
First recall that the proof of continuity of $h(p)$ in $p$ follows easily using the same techniques as in Appendix \ref{subsec:S1}.

	Now note that, if $\|\eta\|_{\infty}\rightarrow\infty$, then there exists $\mathcal X\subseteq\mathfrak X$ such that
	either $E_{\bX}\left[p_0(\bX)\log\left(\frac{p_0(\bX)}{p(\bX)}\right)I_{\mathcal X}\right]\rightarrow\infty$ or
	$E_{\bX}\left[(1-p_0(\bX))\log\left(\frac{1-p_0(\bX)}{1-p(\bX)}\right)I_{\mathcal X}\right]\rightarrow\infty$.
	Hence, $h(p)\rightarrow\infty$ as $\|\eta\|_{\infty}\rightarrow\infty$. Thus, $h(p)$ is a coercive function.

	Since $h(p)$ is continuous and coercive, it follows that $S$ is a compact set.

\end{proof}
In a very similar manner, the following lemma also holds for Poisson model set up.
\begin{lem}\label{PScompact}
	$ S = \{\lambda : h(\lambda) \leq \kappa \}$ is a compact set.
\end{lem}
\begin{proof}
	Again, recall that continuity of $h(\lambda)$ in $\lambda$ can be shown using the same techniques as in Appendix \ref{subsec:PS1}, 
	and it is easily seen that if $\|\eta\|_{\infty}\rightarrow\infty$,
	then $h(\lambda)\rightarrow\infty$. Thus, $h(\lambda)$ is continuous and coercive, ensuring that $S$ is compact.
\end{proof}



Using compactness of $S$, in the same way as in \ctn{Chatterjee18a}, condition (S6) of Shalizi can be shown to be equivalent to (\ref{eq:binary_S6})
and (\ref{eq:Poisson_S6}) in Theorems \ref{mainth1} and \ref{Pmainth1} below, corresponding to binary and Poisson cases. 
In the supplement we show that these equivalent conditions are satisfied in our model setups. 

\begin{theorem}\label{mainth1}
	For the binary regression setup, (S6) is equivalent to the following, which holds under Assumptions \ref{AA1} -- \ref{AA6}:
	\begin{equation}
	\sum_{n=1}^{\infty}\int_{S^c} P\left(\left| \dfrac{1}{n} \log  R_{n}(p) + h(p) \right|> \kappa -h(\Theta) \right) \ d\pi(p) < \infty.
	\label{eq:binary_S6}
	\end{equation}
	
	
\end{theorem} 

\begin{theorem}\label{Pmainth1}
	For the Poisson regression model set up, (S6) is equivalent to the following, which holds  under Assumptions \ref{AA1}--\ref{AA5} and \ref{AA7}:
	\begin{equation}
	\sum_{n=1}^{\infty}\int_{S^c} P\left(\left| \dfrac{1}{n} \log  R_{n}(\lambda) + h(\lambda) \right|> \kappa -h(\Lambda) \right) \ d\pi(\lambda) < \infty.
	\label{eq:Poisson_S6}
	\end{equation}
	
	
\end{theorem} 

Assumption (S7) of Shalizi also holds for both the model setups because of continuity of $h(p)$ and $h(\lambda)$. Hence, all the assumptions (S1)--(S7) stated in 
Appendix \ref{ASZ} are satisfied for binary and Poisson regression setups. 

Overall, our results lead  to the following theorems.

\begin{theorem}\label{mainth2}
	Assume the nonparametric binary regression setup. Then under the Assumptions \ref{AA1}--\ref{AA6},
	\begin{equation}
	\lim_{n \rightarrow \infty} \pi (A| \mathbf{Y}_{n})=0.
	\end{equation}
	Also, for any measurable set A with $\pi(A) > 0$, if $\beta > 2h(A)$, where $h$ is given by equation \eqref{h(p)}, or if $A \subset  \bigcap_{k=n}^{\infty} \mathcal{G}_k$ for some $n$, where $\mathcal{G}_k$
	is given by \ref{sieveset1}, then the followings hold:
	\begin{enumerate}
		
		\item [(i)]\begin{equation}
		\underset{n\rightarrow\infty}{\lim}~\frac{1}{n}\log\left[\pi(A|\mathbf{Y}_n)\right]=-J(A),
		\end{equation}

		\item [(ii)] \begin{equation}
		h(A)>h(\Theta), \pi(A)>0  \ \ \ \Rightarrow  \ \ \lim_{n \rightarrow \infty} \pi\left(A|\mathbf{Y}_n\right) =0.
		\end{equation}
		
	\end{enumerate}
\end{theorem}

\begin{theorem}\label{Pmainth2}
	Assume the nonparametric Poisson regression setup. Then under Assumptions \ref{AA1}--\ref{AA5} and \ref{AA7},
	\begin{equation}
	\lim_{n \rightarrow \infty} \pi (A| \mathbf{Y}_{n})=0.
	\end{equation}
	Also, for any measurable set A with $\pi(A) > 0$, if $\beta > 2h(A)$, where $h$ is given by equation \eqref{Ph(p)}, or if $A \subset  \bigcap_{k=n}^{\infty} \mathbb{G}_k$ for some $n$, where $\mathbb{G}_k$
	is given by \ref{Psieveset1}, then the followings hold:
	\begin{enumerate}
		
		\item [(i)] \begin{equation}
		\underset{n\rightarrow\infty}{\lim}~\frac{1}{n}\log\left[\pi(A|\mathbf{Y}_n)\right]=-J(A),
		\end{equation}

		\item [(ii)] \begin{equation}
		h(A)>h(\Lambda), \pi(A)>0  \ \ \ \Rightarrow  \ \ \lim_{n \rightarrow \infty} \pi\left(A|\mathbf{Y}_n\right) =0.
		\end{equation}
		
	\end{enumerate}
\end{theorem}

\section{ Rate of convergence}
\label{rate}

Consider a sequence of positive reals $\epsilon_n$ such that $\epsilon_n \rightarrow 0$  while $n \epsilon_n \rightarrow \infty$ as $n \rightarrow \infty$ and the set $N_{\epsilon_n}=\{p : h(p)\leq h(\Theta)+\epsilon_n  \}$. Then the following result of Shalizi holds.
\begin{theorem}[\ctn{Shalizi09}]\label{rate1}
	Assume (S1) to (S7) of Appendix \ref{ASZ}. If for each $\delta >0$,
	\begin{equation}
	\tau\left( \mathcal{G}_n \cap  N_{\epsilon_n}^{c}, \delta \right)\leq n
	\end{equation}
	
	eventually almost surely, then almost surely the following holds: 
	\begin{equation}\label{con1}
	\lim_{n \rightarrow \infty}\left(N_{\epsilon_n} |\mathbf{Y}_{n}  \right) =1.
	\end{equation}
	
\end{theorem}

To investigate the rate of convergence in our cases (and also for the case of \ctn{Chatterjee18a}), it has been proved in 
\ctn{Chatterjee18a} that $\epsilon_n$ will be the rate of convergence for $\epsilon_n \rightarrow 0$,   $n \epsilon_n \rightarrow \infty$ as $n \rightarrow \infty$,  if  we can show that the following hold: 

\begin{equation}\label{toshow1}
\dfrac{1}{n} \log \int_{ \mathcal{G}_n \cap  N_{\epsilon_n}^{c}} R_{n}(p) \ d\pi(p) \leq -h(\Theta) +\epsilon,
\end{equation}

\begin{equation}\label{Ptoshow1}
\dfrac{1}{n} \log \int_{ \mathbb{G}_n \cap  N_{\epsilon_n}^{c}} R_{n}(\lambda) \ d\pi(\lambda) \leq -h(\Lambda) +\epsilon,
\end{equation}
for any $\epsilon > 0 $ and all $n$ sufficiently large.


Following similar arguments of \ctn{Chatterjee18a}, 
we find that the posterior rate of convergence with respect to KL-divergence is just slower than $n^{-1}$. To put it another way, 
it is just slower that $n^{-\frac{1}{2}}$ with respect to Hellinger distance for the model setups we consider. 
Our results can be formally stated in Theorem \ref{rate2} for Binary regression and in Theorem \ref{Prate2} for Poisson regression.

\begin{theorem}\label{rate2}
	For the nonparametric binary regression  setup, under Assumptions \ref{AA1}--\ref{AA6}, $\lim_{n \rightarrow \infty}\left(N_{\epsilon_n} |\mathbf{Y}_{n}  \right) =1$  holds almost surely,  where $N_{\epsilon_n}=\{p : h(p)\leq h(\Theta)+\epsilon_n  \}$, $\epsilon_n \rightarrow 0$,   $n \epsilon_n \rightarrow \infty$ as $n \rightarrow \infty$.
\end{theorem}

\begin{theorem}\label{Prate2}
	For the nonparametric Poisson regression setup, under Assumptions \ref{AA1}--\ref{AA5} and \ref{AA7}, $\lim_{n \rightarrow \infty}\left(N_{\epsilon_n} |\mathbf{Y}_{n}  \right) =1$  holds almost surely,  where $N_{\epsilon_n}=\{\lambda : h(\lambda)\leq h(\Lambda)+\epsilon_n  \}$, $\epsilon_n \rightarrow 0$,   $n \epsilon_n \rightarrow \infty$ as $n \rightarrow \infty$.
\end{theorem}

\section{ Consequences of model misspecification}
\label{miss}
Suppose that the true function $\eta_0$ consists of countable number of discontinuities but has continuous first order partial derivatives at all other points. Then $\eta_0 \not \in \mathcal{C'}(\mathfrak{X})$. 
However, there exists some $\tilde{\eta}  \in \mathcal{C'}(\mathfrak{X})$ such that $\tilde{\eta} (x)=\eta_{0}(x)$ for all $x \in \mathfrak{X}$ where $\eta_0$ is continuous. Similar to this  kind of situation is mentioned in \ctn{Chatterjee18a}. Observe that, if the probability measure $Q$ of $X_i$ is dominated by
the Lebesgue measure, then from  Theorem \ref{T1} we have $ h(\Theta) = 0$.  Then  the posterior of $\eta$ concentrates around $\tilde{\eta}$, which
is the same as $\eta_0$ except at the countable number of discontinuities of $\eta_0$. Corresponding $\tilde{p}=H(\tilde{\eta})$ and $\tilde{\lambda}=H(\tilde{\eta})$ will also differ from $p_0$ and $\lambda_0$. If $p_0$ and $\lambda_0$ are such that
$0 < h(\Theta) < \infty$ and $0 < h(\Lambda) < \infty$ respectively 
then the posteriors concentrate around the minimizers of $h(p)$ and $h(\lambda)$, provided such
minimizers exist in $\Theta$ and $\Lambda$, respectively.

\subsection{Consequences from the subjective Bayesian perspective}

Bayesian posterior consistency has two apparently different viewpoints, namely, classical and subjective.  Bayesian analysis starts with a prior knowledge, and updates the 
knowledge given the data, forming the posterior. It is of utmost importance to know whether the updated knowledge becomes more and more accurate and precise as data  
are collected indefinitely.  This requirement is called consistency of the posterior distribution. From the classical Bayesian point of view we should  believe in 
existence of a true model. On the contrary, if we look from the subjective Bayesian viewpoint, then we need not believe in true
models. A subjective Bayesian  thinks only in terms of the predictive distribution of future observations. But \ctn{Blackwell62}, \ctn{Diaconis86} have shown that 
consistency is equivalent to inter subjective agreement, which means that two Bayesians will ultimately have 
very close posterior predictive distributions. 

Let us define the one-step-ahead predictive distribution of $p$ and $\lambda$, one-step-ahead best predictor (which is the best prediction
one could make had the true model, $P$, been known) and the posterior predictive distribution (\ctn{Shalizi09}), with the convention that $n = 1$ gives the marginal distribution 
of the first observation, as follows: 

\begin{itemize}

	\item [](One-step-ahead predictive distribution of $p$):
	$F_{p}^{n}=F_{p} \left(Y_n|Y_1, \ldots, Y_{n-1}\right)$,
	\item [](One-step-ahead predictive distribution of $\lambda$):
	$F_{\lambda}^{n}=F_{\lambda} \left(Y_n|Y_1, \ldots, Y_{n-1}\right)$,
	\item [](One-step-ahead best predictor): 
	$P^{n}=P^{n} \left(Y_n|Y_1, \ldots, Y_{n-1}\right)$,

	\item [] (The posterior predictive distribution): $F_{\pi}^{n}= \int F_{p}^{n} \ d\pi (p |\textbf{Y}_n)$.
	
\end{itemize}

With the above definitions, the following results have been proved by Shalizi. 

\begin{theorem}[\ctn{Shalizi09}]\label{misst1}
	Let  $\rho_H$ and $\rho_{TV}$ be Hellinger and total variation metrics, respectively. Then with probability 1,
	
	\begin{eqnarray*}
	\limsup_{n \rightarrow \infty} \rho_{H}^{2}\left(P^n, F_{\pi}^{n} \right) \leq h(\Theta); \\
	\limsup_{n \rightarrow \infty} \rho_{TV}^{2}\left(P^n, F_{\pi}^{n} \right) \leq 4h(\Theta).
	\end{eqnarray*}
	
\end{theorem}
In our nonparametric setup,
$h(\Theta)=0$ and $h(\Lambda)=0$ if $\eta_0$ consists of countable number of discontinuities. Hence, from Theorem \ref{misst1} it is clear that in spite of 
such misspecification, 
the posterior predictive distribution does a good job in
learning the best possible predictive distribution in terms of the popular Hellinger and the total
variation distance. We state our result formally as follows. 

\begin{theorem}\label{misst2}
	Consider the setups of nonparametric binary and Poisson regression. Assume that the truth function $\eta_0$ consists of countable number of discontinuities but
	has continuous first order partial derivatives at all other points. Then under Assumptions \ref{AA1}--\ref{AA6} (for binary regression) or under 
	Assumptions \ref{AA1}--\ref{AA5} and \ref{AA7} (for Poisson regression) the following hold:
	\begin{eqnarray*}
	\limsup_{n \rightarrow \infty} \rho_{H}^{2}\left(P^n, F_{\pi}^{n} \right)= 0; \\
	\limsup_{n \rightarrow \infty} \rho_{TV}^{2}\left(P^n, F_{\pi}^{n} \right) = 0.
	\end{eqnarray*}
	
\end{theorem}

\section{Conclusion and future work}
\label{conclusion}
In this paper we attempted to address posterior convergence of nonparametric binary and Poisson regression, along with the rate of convergence, 
while also allowing for misspecification, using the approach of \ctn{Shalizi09}. We also have shown that, even in the case
of misspecification, the posterior predictive distribution can be quite accurate asymptotically, which should be a point of interest from subjective Bayesian viewpoint. 
The asymptotic equipartition property plays a central role here. It is one of the crucial assumptions and yet relatively easy to establish under mild conditions. 
It actually  brings forward the KL property of the posterior, which in turn characterizes the
posterior convergence, and also the rate of posterior convergence and misspecification. 

\newpage
\section*{Appendix}
\appendix
\section{ Assumptions and theorems of Shalizi}

\label{ASZ}

Following \ctn{Shalizi09}, let us consider a probability space $(\Omega,\mathcal F,P)$, a sequence
of random variables $\{Y_1,Y_2,\ldots\}$ taking values in the measurable space $(\aleph,\mathcal X)$,
having infinite-dimensional distribution $P$. 
The theoretical development requires no restrictive assumptions on $P$ such as it being a product measure, Markovian, 
or exchangeable, thus paving the way for great generality. 

Let $\mathcal F_n=\sigma(\textbf{Y}_n)$ denote the natural filtration, that is, the $\sigma$-algebra generated by $\textbf{Y}_n$. 
Also, let the distributions of the processes adapted to $\mathcal F_n$ be denoted by $F_{\theta}$, where
$\theta$ takes values in a measurable space $(\Theta,\mathcal T)$. Here $\theta$ denotes the hypothesized probability measure
associated with the unknown distribution of $\{Y_1,Y_2,\ldots\}$ and $\Theta$ is the set of hypothesized
probability measures. 
In other words, assuming that $\theta$ is the
infinite-dimensional distribution of the stochastic process $\{Y_1,Y_2,\ldots\}$, $F_{\theta}$ denotes the 
$n$-dimensional marginal distribution associated with $\theta$; $n$ is suppressed for the ease of notation.
For parametric models, the probability 
measure $\theta$ corresponds to some probability density with respect to some dominating measure (such as Lebesgue
or counting measure) and consists of unknown, but finite number of parameters. 
For nonparametric models, $\theta$ is usually associated with infinite number of parameters and may not even
have any density with respect to $\sigma$-finite measures.

As in \ctn{Shalizi09}, we assume that $P$ and all the $F_{\theta}$ are dominated by a common measure
with densities $p$ and $f_\theta$, respectively. In \ctn{Shalizi09} and in our
case, the assumption that $P\in\Theta$, is not required, so that all possible models are allowed to be misspecified.
Indeed, \ctn{Shalizi09} provides an example of such misspecification where the true model $P$ is not Markov 
but all the hypothesized models indexed by $\theta$ are $k$-th order stationary binary Markov models, for $k=1,2,\ldots$. 
As shown in \ctn{Shalizi09}, the results of posterior convergence hold even in the case of such misspecification, 
essentially because the true model can be approximated by the $k$-th order Markov models belonging to $\Theta$.

Given a prior $\pi$ on $\theta$, we assume that the posterior distributions $\pi(\cdot|\textbf{Y}_n)$ are dominated by a common
measure for all $n>0$. 

\subsection{Assumptions}
\label{subsec:assumptions_shalizi}

\begin{itemize}
	\item[(S1)] Letting $f_{\theta}(\mathbf{Y}_n)$ be the likelihood under parameter $\theta$ an$f_{\theta_0}(\mathbf{Y}_n)$ be the likelihood under 
		the true parameter $\theta_0$, given the true model $P$, consider the following likelihood ratio:
	\begin{equation}
	R_n(\theta)=\frac{f_{\theta}(\mathbf{Y}_n)}{f_{\theta_0}(\mathbf{Y}_n)}.
	\label{eq:R_T}
	\end{equation}
	Assume that $R_n(\theta)$ is $\mathcal F_n\times \mathcal T$-measurable for all $n>0$.
\end{itemize}

\begin{itemize}
	\item[(S2)] For every $\theta\in\Theta$, the KL divergence rate
	\begin{equation}
	h(\theta)=\underset{n\rightarrow\infty}{\lim}~\frac{1}{n}E\left[\log\left\{\frac{f_{\theta_0}(\mathbf{Y}_n)}{f_{\theta}(\mathbf{Y}_n)}\right\}\right].
	\label{eq:S2}
	\end{equation}
	exists (possibly being infinite) and is $\mathcal T$-measurable.
	Note that in the $iid$ set-up, $h(\theta)$ reduces to the KL divergence between the true and the
	hypothesized model, so that (\ref{eq:S2}) may be regarded as a generalized KL divergence measure.
\end{itemize}

\begin{itemize}
	\item[(S3)] For each $\theta\in\Theta$, the generalized or relative asymptotic equipartition property holds, and so,
	almost surely with respect to $P$,
	\begin{equation}
	\underset{n\rightarrow\infty}{\lim}~\frac{1}{n}\log \left[R_n(\theta)\right]=-h(\theta),
	\label{eq:equipartition}
	\end{equation}
	where $h(\theta)$ is given by (\ref{eq:S2}).
	
	Intuitively, the terminology 
	``asymptotic equipartition" refers to dividing up $\log \left[R_n(\theta)\right]$ into $n$ factors for large $n$ 
	such that all the factors are asymptotically equal. Again, considering the $iid$ scenario helps clarify this point,
	as in this case each factor converges to the same KL divergence between the true and the postulated model. 
	With this understanding note that the purpose of condition (S3) is to ensure that relative to the true distribution, 
	the likelihood of each $\theta$ decreases to zero exponentially
	fast, with rate being the KL divergence rate (\ref{eq:equipartition}). 
\end{itemize}

\begin{itemize}
	\item[(S4)] 
	Let $I=\left\{\theta:h(\theta)=\infty\right\}$. 
	The prior $\pi$ on $\theta$ satisfies $\pi(I)<1$.
	Failure of this assumption entails extreme misspecification of almost all the hypothesized models $f_{\theta}$ relative
	to the true model $p$. With such extreme misspecification, posterior consistency is not expected to hold.
\end{itemize}

\begin{itemize}
	\item[(S5)] There exists a sequence of sets $\mathcal G_n\rightarrow\Theta$ as $n\rightarrow\infty$ 
	such that: 
	\begin{enumerate}
		\item $h\left(\mathcal G_n\right)\rightarrow h\left(\Theta\right)$, as $n\rightarrow\infty$.
		\item The following inequality holds for some $\alpha>0, \beta> 2h(\Theta)$ 
		\begin{equation*}
		\pi\left(\mathcal G_n\right)\geq 1-\alpha\exp\left(-\beta n\right);
		\end{equation*}
		\item The convergence in (S3) is uniform in $\theta$ over $\mathcal G_n\setminus I$.
	\end{enumerate}
	The sets $\mathcal G_n$ can be loosely interpreted as the sieves. Method of sieves is common to Bayesian non parametric approach,
	such that the behaviour of the likelihood ratio and the posterior on the sets $\mathcal G_n$ 
	essentially carries over to $\Theta$. 
	This can be anticipated from the first and the second parts of the assumption; the second part ensuring in particular
	that the parts of $\Theta$ on which the log likelihood ratio may be ill-behaved have exponentially small prior probabilities.
	The third part is more of a technical condition that is useful in proving posterior convergence through the sets
		$\mathcal G_n$. For further details, see \ctn{Shalizi09}.
\end{itemize}

For each measurable $A\subseteq\Theta$, for every $\delta>0$, there exists a random natural number $\tau(A,\delta)$
such that
\begin{equation}
\frac{1}{n}\log\left[\int_{A}R_n(\theta)\pi(\theta)d\theta\right]
\leq \delta+\underset{n\rightarrow\infty}{\lim\sup}~\frac{1}{n}
\log\left[\int_{A}R_n(\theta)\pi(\theta)d\theta\right],
\label{eq:limsup_2}
\end{equation}
for all $n>\tau(A,\delta)$, provided 
$\underset{n\rightarrow\infty}{\lim\sup}~\frac{1}{n}\log\left[\int_{A}R_n(\theta)\pi(\theta)d\theta\right]<\infty$.
Regarding this, the following assumption has been made by Shalizi:
\begin{itemize}
	\item[(S6)] The sets $\mathcal G_n$ of (A5) can be chosen such that for every $\delta>0$, the inequality
	$n>\tau(\mathcal G_n,\delta)$ holds almost surely for all sufficiently large $n$.
	
	To understand the essence of this assumption, note that for almost every data set $\{Y_1,Y_2,\ldots\}$
	there exists $\tau(\mathcal G_n,\delta)$ such that equation (\ref{eq:limsup_2}) holds with $A$ replaced by $\mathcal G_n$
	for all $n>\tau(\mathcal G_n,\delta)$. Since $\mathcal G_n$ are sets with large enough prior probabilities,
	the assumption formalizes our expectation that $R_n(\theta)$ decays fast enough on $\mathcal G_n$ so that 
	$\tau(\mathcal G_n,\delta)$ is nearly stable in the sense that it is not only finite but also not 
	significantly different for different data sets when $n$ is large. See \ctn{Shalizi09} for more detailed explanation.
	
\end{itemize}

\begin{itemize}
	\item[(S7)]The sets $\mathcal G_n$ of (S5) and (S6) can be chosen such that for any set $A$ with $\pi(A) > 0$, 
	\begin{equation}
	\lim_{n \rightarrow \infty}  h\left( \mathcal G_n \cap A \right)= h(A).
	\end{equation}
\end{itemize}

Under the above assumptions, \ctn{Shalizi09} proved the following results.

\begin{theorem}[\ctn{Shalizi09}]
	\label{theorem:shalizi1}
	Consider assumptions (S1)--(S7) and any set $A\in\mathcal T$ with $\pi(A)>0$ and $h(A)>h(\Theta)$. Then,
	\begin{equation*}
		\underset{n\rightarrow\infty}{\lim}~\pi(A|\bY_n)=0,~\mbox{almost surely}.
	\end{equation*}
\end{theorem}

The rate of convergence of the log-posterior is given by the following result. 
\begin{theorem}[\ctn{Shalizi09}]
	\label{theorem:shalizi2}
	Consider assumptions (S1)--(S7) and any set $A\in\mathcal T$ with $\pi(A)>0$. If $\beta>2h(A)$, where
	$\beta$ corresponds to assumption (S5), or if $A\subset\cap_{k=n}^{\infty}\mathcal G_k$ for some $n$, then
	\begin{equation*}
		\underset{n\rightarrow\infty}{\lim}~\frac{1}{n}\log\pi(A|\bY_n)=-J(A),~\mbox{almost surely.}
	\end{equation*}
\end{theorem}


\newpage

\section{ Verification of (S1) to (S7) for binary regression}\label{Shaverify}

\subsection{Verification of (S1) for binary regression }
\label{subsec:S1}

Observe that
\begin{align}
& f_{p}(\bY_n|\bX_n)=\prod_{i=1}^{n} f(y_i|x_i) = \prod_{i=1}^{n}  p(x_i)^{y_i} \left(1-p(x_i)\right)^{1-y_i},  \\
& f_{p_0}(\bY_n|\bX_n)=\prod_{i=1}^{n} f_{0}(y_i|x_i) = \prod_{i=1}^{n}  p_{0}(x_i)^{y_i} \left(1-p_{0}(x_i)\right)^{1-y_i}. 
\end{align}

Therefore,

\begin{align} \frac{1}{n} \log R_{n}(p)= &\frac{1}{n}\sum_{i=1}^{n} \left\lbrace\left( y_i\log \left(\frac{p(x_i)}{p_0(x_i)}\right) \right)
	+ (1-y_i) \log \left(\frac{1- p(x_i)}{1-p_0(x_i)}\right) \right \rbrace. \label{S11}
\end{align}

To show measurability of $R_n(p)$, first note that for any $a\in\mathfrak R$,
\begin{align}
	&\left\{(y_i,\eta): y_i\log \left(\frac{p(x_i)}{p_0(x_i)}\right) 
	+ (1-y_i) \log \left(\frac{1- p(x_i)}{1-p_0(x_i)}\right) < a\right\}\notag\\
	&=\left\{\eta:\log \left(\frac{p(x_i)}{p_0(x_i)}\right)<a\right\}\bigcup\left\{\eta:\log \left(\frac{1-p(x_i)}{1-p_0(x_i)}\right)<a\right\}.
\label{eq:S1}
\end{align}
Note that for given $p$, there exists $0<\epsilon<1/2$ such that $\epsilon< p(x) <1-\epsilon$, for all $x\in\mathfrak X$.
Now consider a sequence $\tilde\eta_j$, $j=1,2,\ldots$ such that $\|\tilde\eta_j-\eta\|_{\infty}\rightarrow 0$, as $j\rightarrow\infty$. Then, with 
$\tilde p_j(x)=H\left(\tilde\eta_j(x)\right)$, note that there exists $j_0\geq 1$ such that for $j\geq j_0$,  
$\epsilon<\tilde p_j(x) <1-\epsilon$, for all $x\in\mathfrak X$. Hence, using the inequality $1-\frac{1}{x}\leq\log x\leq x-1$ for $x>0$, we obtain
$\left|\log\left(\frac{\tilde p_j(x_i)}{p(x_i)}\right)\right|\leq C\|\tilde p_j-p\|_{\infty}$ and 
$\left|\log\left(\frac{1-\tilde p_j(x_i)}{1-p(x_i)}\right)\right|\leq C\|\tilde p_j-p\|_{\infty}$, for some $C>0$, for all $x\in\mathfrak X$.
Hence, for $j\geq j_0$,
\begin{align}
&\left|\log\left(\frac{\tilde p_j(x_i)}{p_0(x_i)}\right)-\log\left(\frac{p(x_i)}{p_0(x_i)}\right)\right|
	=\left|\log\left(\frac{\tilde p_j(x_i)}{p(x_i)}\right)\right|\leq C\|\tilde p_j-p\|_{\infty}. 
	\label{eq:taylor0}
\end{align}
Now, since $H$ is continuously differentiable, using Taylor's series expansion up to the first order we obtain, 
\begin{align}
	\|\tilde p_j-p\|_{\infty}&=\underset{x\in\mathfrak X}{\sup}~\left|H\left(\tilde\eta_j(x)\right)-H\left(\eta(x)\right)\right|\notag\\
	&=\underset{x\in\mathfrak X}{\sup}~\left|H'(u(\tilde\eta_j(x),\eta(x)))\right|\|\tilde\eta_j-\eta\|_{\infty},
	\label{eq:taylor1}
\end{align}
where $u(\tilde\eta_j(x),\eta(x))$ lies between $\eta(x)$ and $\tilde\eta_j(x)-\eta(x)$. Since $\|\tilde\eta_j-\eta\|_{\infty}\rightarrow 0$, as $j\rightarrow\infty$, 
it follows from (\ref{eq:taylor1}) that $\|\tilde p_j-p\|_{\infty}\rightarrow 0$, as $j\rightarrow\infty$.
This again implies, thanks to (\ref{eq:taylor0}), that $\left|\log\left(\frac{\tilde p_j(x_i)}{p_0(x_i)}\right)-\log\left(\frac{p(x_i)}{p_0(x_i)}\right)\right|
\rightarrow 0$, as $j\rightarrow\infty$.

In other words, $\log\left(\frac{p(x_i)}{p_0(x_i)}\right)$ is continuous
in $\eta$, and hence $\left\{\eta:\log \left(\frac{p(x_i)}{p_0(x_i)}\right)<a\right\}$ of (\ref{eq:S1}) is measurable. 
Similarly, $\log\left(\frac{1-p(x_i)}{1-p_0(x_i)}\right)$ is also continuous in $\eta$, so that 
$\left\{\eta:\log \left(\frac{1-p(x_i)}{1-p_0(x_i)}\right)<a\right\}$ is also measurable. Hence, 
the individual terms in (\ref{S11}) are measurable. 
Since sums of measurable functions are measurable, it follows that $\log R_n(p)$, and hence $R_n(p)$, is measurable.



\subsection{Verification of (S2) for binary regression}
\label{subsec:S2}
for every $p\in \Theta$, we need to show that  the KL divergence rate

\begin{equation*}
h(p)=\underset{n\rightarrow\infty}{\lim}~\frac{1}{n}E_{p_0}\left[\log\left\{\frac{f_{p_0}(\bY_n|\bX_n)}{f_{p}(\bY_n|\bX_n)}\right\}\right] =\underset{n\rightarrow\infty}{\lim}~\frac{1}{n}E_{p_0}\left[-\log\left\{R_{n}(p)\right\}\right].
\end{equation*}
exists (possibly being infinite) and is $\mathcal T$-measurable.

Now,
\begin{align} \frac{1}{n} \log R_{n}(p)= &  \frac{1}{n}\sum_{i=1}^{n} \left\lbrace\left( y_i\log p(x_i) \right)+ (1-y_i) \log \left(1- p(x_i) \right) \right \rbrace \\
- & \frac{1}{n}\sum_{i=1}^{n} \left\lbrace\left( y_i\log p_{0}(x_i) \right)+ (1-y_i) \log \left(1- p_{0}(x_i) \right) \right \rbrace .\notag
\end{align}
Therefore,

\begin{align}\frac{1}{n}E_{p_0}\left[-\log\left\{R_{n}(p)\right\}\right]=  &\frac{1}{n}\sum_{i=1}^{n} \left\lbrace\left(p_{0}(x_i) \log p_{0}(x_i) \right)+ (1-p_{0}(x_i)) \log \left(1- p_{0}(x_i) \right) \right \rbrace \\
& - \frac{1}{n}\sum_{i=1}^{n} \left\lbrace\left(p_{0}(x_i)\log p(x_i) \right)+ (1-p_{0}(x_i)) \log \left(1- p(x_i) \right) \right \rbrace.\notag
\end{align}

\begin{align}\underset{n\rightarrow\infty}{\lim}~\frac{1}{n}E_{p_0}\left[-\log\left\{R_{n}(p)\right\}\right]=  &\underset{n\rightarrow\infty}{\lim}~\frac{1}{n}\sum_{i=1}^{n} \left\lbrace\left(p_{0}(x_i) \log p_{0}(x_i) \right)+ (1-p_{0}(x_i)) \log \left(1- p_{0}(x_i) \right) \right \rbrace\notag \\
& - \underset{n\rightarrow\infty}{\lim}~\frac{1}{n}\sum_{i=1}^{n} \left\lbrace\left(p_{0}(x_i)\log p(x_i) \right)+ (1-p_{0}(x_i)) \log \left(1- p(x_i) \right) 
	\right \rbrace\notag \\
= & E_{\mathbf{X}}\left\lbrace\left( p_{0}(\mathbf{X})\log p_{0}(\mathbf{X}) \right)+ (1-p_{0}(\mathbf{X})) \log \left(1- p_{0}(\mathbf{X}) \right) \right \rbrace\notag\\
-&E_{\mathbf{X}} \left\lbrace\left( p_{0}(\mathbf{X})\log p(\mathbf{X}) \right)+ (1-p_{0}(\mathbf{X})) \log \left(1- p(\mathbf{X}) \right) \right \rbrace. \label{A2221}
\end{align}

The last line follows from Assumption \ref{AA4} and SLLN. Here $E_{\mathbf{X}} (\cdot)=\int_{\mathfrak{X}} \cdot \ dQ$.

Hence,
\begin{equation}
h(p)=\left[E_{\mathbf{X}} \left( p_{0}(\mathbf{X})\log \left\lbrace \dfrac{p_{0}(\mathbf{X})}{p(\mathbf{X})}\right\rbrace \right)+ E_{\mathbf{X}}\left((1-p_{0}(\mathbf{X})) \log \left\lbrace\dfrac{ \left(1- p_{0}(\mathbf{X})\right)}{\left(1- p(\mathbf{X})\right)} \right\rbrace\right) \right ]. \label{h_p}
\end{equation}
%
%

\subsection{Verification of (S3) for binary regression}
\label{subsec:S3}
Here we need to verify the asymptotic equipartition, that is, almost surely with respect to $P$,
\begin{equation}
\underset{n\rightarrow\infty}{\lim}~\frac{1}{n}\log \left[R_n(p)\right]=-h(p)=\underset{n\rightarrow\infty}{\lim}~
	\frac{1}{n}E\left[\log\left\{\frac{f_{p}(\bY_n|\bX_n)}{f_{p_0}(\bY_n|\bX_n)}\right\}\right].
\label{eq:equipartition22}
\end{equation}
Observe that,

\begin{align*} \frac{1}{n} \log R_{n}(p)= &\frac{1}{n}\sum_{i=1}^{n} \left\lbrace\left( y_i\log p(x_i) \right)+ (1-y_i) \log \left(1- p(x_i) \right) \right \rbrace \\
- &\frac{1}{n}\sum_{i=1}^{n} \left\lbrace\left( y_i\log p_{0}(x_i) \right)+ (1-y_i) \log \left(1- p_{0}(x_i) \right) \right \rbrace.
\end{align*}
By rearranging the terms we get,
\begin{align*}- \frac{1}{n} \log R_{n}(p)= &\frac{1}{n}\sum_{i=1}^{n} \left\lbrace y_i\log\left(\dfrac{ p_{0}(x_i)}{p(x_i)}\right) + (1-y_i)\log\left(\dfrac{ 1- p_{0}(x_i)}{1-p(x_i)}\right) \right \rbrace.
\end{align*}

Using the inequality $1-\frac{1}{x}\leq\log x\leq x-1$ for $x>0$, compactness of $\mathfrak X$, and continuity of $p(x)$ in $x\in\mathfrak X$ for 
given $p \in \Theta$, $\left|\log\left(\frac{p_0(x_i)}{p(x_i)}\right)\right|\leq C\|p-p_0\|_{\infty}$ and 
$\left|\log\left(\frac{1-p_0(x_i)}{1-p(x_i)}\right)\right|\leq C\|p-p_0\|_{\infty}$, for some $C>0$. Hence,
\begin{align}
& \displaystyle \sum_{i=1}^{\infty} i^{-2} var \left[ \left\lbrace y_i\log\left(\dfrac{ p_{0}(x_i)}{p(x_i)}\right) + (1-y_i)\log\left(\dfrac{ 1- p_{0}(x_i)}{1-p(x_i)}\right) \right \rbrace \right] \\
	&=\displaystyle\sum_{i=1}^{\infty} i^{-2} p_{0}(x_i)(1-p_0(x_i))\notag\\
	&\qquad\qquad\times\left\lbrace\left[\log\left(\dfrac{ p_{0}(x_i)}{p(x_i)}\right)\right]^2 
	+ \left[\log\left(\dfrac{ 1- p_{0}(x_i)}{1-p(x_i)}\right)\right]^2-2\log\left(\dfrac{ p_{0}(x_i)}{p(x_i)}\right)\times
	\log\left(\dfrac{ 1- p_{0}(x_i)}{1-p(x_i)}\right)\right \rbrace\notag \\
	&\leq 4C^2\|p_0\|_{\infty}\|p-p_0\|^2_{\infty}\sum_{i=1}^{\infty}i^{-2}\notag\\
	&<\infty.
\end{align}

Observe that $y_i$ are  observations from independent random variables.  Hence by Kolmogorov's SLLN for independent random variables, 
\begin{align*}
&-\frac{1}{n} \log R_{n}(p)= \frac{1}{n}\sum_{i=1}^{n} \left\lbrace y_i\log\left(\dfrac{ p_{0}(x_i)}{p(x_i)}\right) + (1-y_i)\log\left(\dfrac{ 1- p_{0}(x_i)}{1-p(x_i)}\right) \right \rbrace \\
	&  \rightarrow  \left[E_{\mathbf{X}} \left( p_{0}(\mathbf{X})\log \left\lbrace \dfrac{p_{0}(\mathbf{X})}{p(\mathbf{X})}\right\rbrace \right)+ E_{\mathbf{X}}\left((1-p_{0}(\mathbf{X})) \log \left\lbrace\dfrac{ \left(1- p_{0}(\mathbf{X})\right)}{\left(1- p(\mathbf{X})\right)} \right\rbrace\right) \right ] =h(p),
\end{align*}
almost surely, as $n\rightarrow\infty$.

\subsection{Verification of (S4) for binary regression }
\label{subsec:S4}

If $I=\{p: \ h(p)=\infty\}$ then we need to show $\Pi(I)<1$. 
Note that due to compactness of $\mathfrak X$ and continuity of $H$ and $\eta$, given $\eta\in\Theta$, $p$ is bounded away from $0$ and $1$. 
Hence, $h(p)\leq\|p-p_0\|_{\infty}\times\left(\frac{1}{\underset{x\in\mathfrak X}{\inf}~p(x)}
+\frac{1}{1-\underset{x\in\mathfrak X}{\sup}~p(x)}\right)<\infty$, almost surely. 
In other words, (S4) holds. 


\subsection{Verification of (S5) for binary regression}
\label{subsec:S5}
In our model, the parameter space is $\Theta=\mathcal{C'}(\mathfrak{X})$. 
We need to show that there exists a sequence of sets $\mathcal G_n\rightarrow\Theta$ as $n\rightarrow\infty$ 
such that: 
\begin{enumerate}
	\item $h\left(\mathcal G_n\right)\rightarrow h\left(\Theta\right)$, as $n\rightarrow\infty$.
	\item The  inequality $\pi\left(\mathcal G_n\right)\geq 1-\alpha\exp\left(-\beta n\right)$ holds for some $\alpha>0, \beta> 2h(\Theta)$. 
	
	\item The convergence in (S3) is uniform in $p$ over $\mathcal G_n\setminus I$.
\end{enumerate}

We shall work with the following sequence of sieve sets considered in \ctn{Chatterjee18a}: for $n\geq 1$,

\begin{equation}\label{sieveset}
\begin{aligned}
	\mathcal{G}_n=\left \lbrace \eta\in\mathcal{C}'(\mathfrak{X}): \  \|\eta\|_{\infty} \leq \exp(\left(\beta n\right)^{1/4}), \  
	\|\eta_{j}'\|_{\infty} \leq \exp(\left(\beta n\right)^{1/4}); j=1, 2, \ldots, d\right \rbrace.
\end{aligned}
\end{equation}


Then  $\mathcal{G}_n \rightarrow \mathcal{C}'(\mathfrak{X})$ as $n \rightarrow \infty$ (\ctn{Chatterjee18a}).

\subsubsection{ Verification of (S5) (1)}
\label{subsubsec:S51}

We now verify that $h\left(\mathcal G_n\right)\rightarrow h\left(\Theta\right)$, as $n\rightarrow\infty$. Observe that:
\begin{equation}
h(p)=\left[E_{\mathbf{X}} \left( p_{0}(\mathbf{X})\log \left\lbrace \dfrac{p_{0}(\mathbf{X})}{p(\mathbf{X})}\right\rbrace \right)+ E_{\mathbf{X}}\left((1-p_{0}(\mathbf{X})) \log \left\lbrace\dfrac{ \left(1- p_{0}(\mathbf{X})\right)}{\left(1- p(\mathbf{X})\right)} \right\rbrace\right) \right ].
\end{equation}\\
Recall that $h(p)$ is continuous in $p$ and $p$ is continuous in $\eta$, which follows from (\ref{eq:taylor1}). 
Hence, continuity of $h(p)$, compactness of $\mathcal{G}_n$ along with its non-decreasing
nature with respect to $n$ implies that $h\left(\mathcal G_n\right)\rightarrow h\left(\Theta\right)$, as $n\rightarrow\infty$.

\subsubsection{ Verification of (S5) (2)}
\label{subsubsec:S52}

\begin{align*}
	\pi(\mathcal{G}_{n}) &=\Pi\left(\|\eta\| \leq \exp(\left(\beta n\right)^{1/4})  
\right)\\
	& -\pi\left(\|\eta_{j}'\| \leq \exp(\left(\beta n\right)^{1/4}); j=1, 2, \ldots, d   \right) \\
	&=\pi\left(\|\eta\| \leq \exp(\left(\beta n\right)^{1/4}),  
 \right)\\
	& -\pi\left(   \bigcup_{j=1}^{d} \left\lbrace\|\eta_{j}'\| \leq \exp(\left(\beta n\right)^{1/4})  \right \rbrace \right) \\
	& \geq 1- \Pi\left(\|\eta\| > \exp(\left(\beta n\right)^{1/4})\right)- \sum_{j=1}^{d} \Pi\left(     \|\eta_{j}'\| \leq \exp(\left(\beta n\right)^{1/4})   \right) \\
& \geq 1- \left(c_{\eta}+ \sum_{j=1}^{d} c_{\eta_{j}'} \right) \exp(-\beta n). 
\end{align*}
where the last inequality follows from Assumption \ref{AA3}.

\subsubsection{ Verification of (S5) (3)}
\label{subsubsec:S53}

We need to show that  uniform convergence in (S3) in $p $ over $\mathcal{G}_n \setminus I$ holds, where $I=\{p: \ h(p)=\infty\}$ as in subsection \ref{subsec:S4}. In our case, $I=\emptyset$. Hence, we need to show  uniform convergence in (S3) in $p$ over $\mathcal{G}_n $. We need to establish that $\mathcal{G}_n$ is compact, but this has
already been shown by \ctn{Chatterjee18a}. In a nutshell, 
\ctn{Chatterjee18a} proved compactness of $\mathcal{G}_{n}$ for each $n \geq 1$ by showing that $\mathcal{G}_{n}$ is closed, bounded and equicontinuous and then by 
using  Arzela-Ascoli lemma to imply compactness. It should be noted that boundedness of the partial derivatives as in Assumption \ref{AA1} is used to show 
Lipschitz continuity, hence equicontinuity. 


Consider $ \mathcal{G} \in  \left \lbrace  \mathcal{G}_n : \  n = 1, 2, \ldots \right\rbrace $. Now, to show uniform convergence we only need to show the following 
(see, for example, \ctn{Chatterjee18a}):

\begin{enumerate}
	
	\item [(i)]$\dfrac{1}{n} \log( R_n(p)) + h(p)$ is stochastically equicontinuous almost surely in $p \in \mathcal{G}$,
	\item [(ii)] $\dfrac{1}{n} \log( R_n(p)) + h(p)\rightarrow 0$ for all $p \in \mathcal{G}$ as $n \rightarrow \infty$.
\end{enumerate}

We have already shown almost sure pointwise convergence of $n^{-1}\log( R_n(p))$ to  $- h(p)$ in Appendix \ref{subsec:S3}. Hence it is enough to verify stochastic equicontinuity of $\dfrac{1}{n} \log( R_n(p)) + h(p)$ in 
$ \mathcal{G} \in  \left \lbrace  \mathcal{G}_n : \  n = 1, 2, \ldots \right\rbrace$. Stochastic equicontinuity usually follows easily if one can prove that the function concerned  is almost surely Lipschitz continuous (\ctn{Chatterjee18a}). 
Observe that, if we can show that both $\dfrac{1}{n} \log( R_n(p))$ and  $ h(p)$ are Lipschitz then this would imply that $\dfrac{1}{n} \log( R_n(p)) + h(p)$ is Lipschitz (sum of Lipschitz functions is Lipschitz). 

We now show that $\dfrac{1}{n} \log( R_n(p)) $ and $h(p)$ are both Lipschitz in $\mathcal G$. Now, 
\begin{align} \frac{1}{n}  \log R_{n}(p)= \dfrac{1}{n}  \sum_{i=1}^{n} \left\lbrace y_i\log\left(\dfrac{ p(x_i)}{p_{0}(x_i)}\right) + (1-y_i)\log\left(\dfrac{ 1- p(x_i)}{1-p_{0}(x_i)}\right) \right \rbrace.
\end{align}

Let $p_1,p_2$ correspond to $\eta_1,\eta_2\in\Theta$. Note that, since $\|\eta\|_{\infty}\leq\exp\left(\sqrt{\beta m}\right)$ on $\mathcal G=\mathcal G_m$ ($m\geq 1$), 
it follows that $0<\kappa_B\leq p_1(x),p_2(x)\leq 1-\kappa_B<1$, for all $x\in\mathfrak X$. Thus, there exists $C>0$ such that 
$\left|\log\left(\frac{p_1(x)}{p_2(x)}\right)\right|\leq C\|p_1-p_2\|_{\infty}$ and $\left|\log\left(\frac{1-p_1(x)}{1-p_2(x)}\right)\right|\leq C\|p_1-p_2\|_{\infty}$,
for $x\in\mathfrak X$.
Hence, 
\begin{align}
&\left|\frac{1}{n}\log R_n(p_1)-\frac{1}{n}\log R_n(p_2)\right|\notag\\
	&\qquad=\left|\frac{1}{n}\sum_{i=1}^n\left\{y_i\log\left(\frac{p_1(x_i)}{p_2(x_i)}\right)+(1-y_i)\log\left(\frac{1-p_1(x_i)}{1-p_2(x_i)}\right)\right\}\right|\notag\\
	&\qquad\leq 2C\|p_1-p_2\|_{\infty},\notag
\end{align}
showing Lipschitz continuity of $\frac{1}{n}\log R_n(p)$ with respect to $p$ corresponding to $\eta\in\mathcal G=\mathcal G_m$.
Since $H$ is continuously differentiable, $\eta$ and $\eta'$ are bounded on $\mathcal G$, with the same bound for all $\eta$, it follows that 
$p$ is Lipschitz on $\mathcal G$.

To see that $h(p)$ is also Lipschitz in $\mathcal G=\mathcal G_m$, it is enough to note that
\begin{align}
&\left|h(p_1)-h(p_2)\right|=\left|E_{\bX}\left(p_0(\bX)\log\left(\frac{p_2(\bX)}{p_1(\bX)}\right)\right)
	+E_{\bX}\left((1-p_0(\bX))\log\left(\frac{1-p_2(\bX)}{1-p_1(\bX)}\right)\right)\right|\notag\\
	&\qquad\leq2C\|p_1-p_2\|_{\infty},\notag
	\end{align}
and the result follows since $p$ is Lipschitz on $\mathcal G$.

\subsection{Verification of (S6) for binary regression}
\label{subsec:S6}
We need to show:
\begin{equation}
\sum_{n=1}^{\infty}\int_{S^c} P\left(\left| \dfrac{1}{n} \log  R_{n}(p) + h(p) \right|> \kappa -h(\Theta) \right) \ d\pi(p) < \infty.
\label{eq:finite1}
\end{equation}
Let us take $\kappa_1=\kappa -h(\Theta)$.
Observe that, 
\begin{align*} &\frac{1}{n}  \log R_{n}(p) +h(p)\\
=& \dfrac{1}{n}  \sum_{i=1}^{n} \left\lbrace y_i\log\left(\dfrac{ p(x_i)}{p_{0}(x_i)}\right) + (1-y_i)\log\left(\dfrac{ 1- p(x_i)}{1-p_{0}(x_i)}\right) \right \rbrace \\
&+ \left[E_{\mathbf{X}} \left( p_{0}(\mathbf{X})\log \left\lbrace \dfrac{p_{0}(\mathbf{X})}{p(\mathbf{X})}\right\rbrace \right)+ E_{\mathbf{X}}\left((1-p_{0}(\mathbf{X})) \log \left\lbrace\dfrac{ \left(1- p_{0}(\mathbf{X})\right)}{\left(1- p(\mathbf{X})\right)} \right\rbrace\right) \right ]\\
=& \dfrac{1}{n}  \sum_{i=1}^{n} \left\lbrace y_i\log\left(\dfrac{ p(x_i)}{p_{0}(x_i)}\right) - E_{\mathbf{X}} \left( p_{0}(\mathbf{X})\log \left\lbrace \dfrac{p(\mathbf{X})}{p_{0}(\mathbf{X})}\right\rbrace \right)\right\rbrace \\
& + \dfrac{1}{n}  \sum_{i=1}^{n} \left\lbrace(1-y_i)\log\left(\dfrac{ 1- p(x_i)}{1-p_{0}(x_i)}\right)-E_{\mathbf{X}}\left((1-p_{0}(\mathbf{X})) \log \left\lbrace\dfrac{ \left(1- p(\mathbf{X})\right)}{\left(1- p_{0}(\mathbf{X})\right)} \right\rbrace\right) \right \rbrace.
\end{align*}

It follows that:
\begin{align} & P\left(\left| \dfrac{1}{n} \log  R_{n}(p) + h(p) \right|> \kappa_1 \right)  \\
&\leq P\left(\left| \dfrac{1}{n}  \sum_{i=1}^{n} \left\lbrace y_i\log\left(\dfrac{ p(x_i)}{p_{0}(x_i)}\right) - E_{\mathbf{X}} \left( p_{0}(\mathbf{X})\log \left\lbrace \dfrac{p(\mathbf{X})}{p_{0}(\mathbf{X})}\right\rbrace \right)\right\rbrace \right|> \dfrac{\kappa_1}{2} \right) \\
& + P\left(\left|  \dfrac{1}{n}  \sum_{i=1}^{n} \left\lbrace(1-y_i)\log\left(\dfrac{ 1- p(x_i)}{1-p_{0}(x_i)}\right)-E_{\mathbf{X}}\left((1-p_{0}(\mathbf{X})) \log \left\lbrace\dfrac{ \left(1- p(\mathbf{X})\right)}{\left(1- p_{0}(\mathbf{X})\right)} \right\rbrace\right) \right \rbrace  \right|> \dfrac{\kappa_1}{2} \right).\label{hoef12}
\end{align}

Since $y_i$ are binary, it follows using the inequalities $1-\frac{1}{x}\leq\log x\leq x-1$, for $x>0$ and Assumptions 
\ref{AA5} and \ref{AA6}, that the random variables $V_i= y_i\log\left(\dfrac{ p(x_i)}{p_{0}(x_i)}\right)$ 
and $W_i=  y_i\log\left(\dfrac{ 1-p(x_i)}{1-p_{0}(x_i)}\right)$ are absolutely bounded by $C\|p-p_0\|_{\infty}$,
for some $C>0$. 
We shall apply Hoeffding's inequality (\ctn{Hoeffding63}) separately on the two terms of \eqref{hoef12} involving $V_i$ and $W_i$.

Note that for $\eta\in\mathcal G_n$,
\begin{align}
& P\left(\left| \dfrac{1}{n}  \sum_{i=1}^{n} \left\lbrace y_i\log\left(\dfrac{ p(x_i)}{p_{0}(x_i)}\right) - E_{\mathbf{X}} \left( p_{0}(\mathbf{X})\log \left\lbrace \dfrac{p(\mathbf{X})}{p_{0}(\mathbf{X})}\right\rbrace \right)\right\rbrace \right|> \dfrac{\kappa_1}{2} \right)\notag \\
& \leq  
	P\left(\left| \dfrac{1}{n}  \sum_{i=1}^{n} \left\lbrace y_i\log\left(\dfrac{ p(x_i)}{p_{0}(x_i)}\right) - 
	p_0(x_i)\log\left(\frac{p(x_i)}{p_0(x_i)}\right)\right\rbrace\right| > \dfrac{\kappa_1}{4} \right)\notag \\
	&\qquad+ P\left(\left| \dfrac{1}{n}  \sum_{i=1}^{n}\left\{p_0(x_i)\log\left(\frac{p(x_i)}{p_0(x_i)}\right)
	-E_{\mathbf{X}} \left( p_{0}(\mathbf{X})\log \left\lbrace \dfrac{p(\mathbf{X})}{p_{0}(\mathbf{X})}\right\rbrace\right)\right\}\right| > \dfrac{\kappa_1}{4} \right) 
	\notag\\
	&\leq 4\exp\left\lbrace  -\dfrac{n\kappa^2_1}{8C^2\|p-p_0\|^2_{\infty}} \right \rbrace
	\leq 4\exp\left\lbrace  -\dfrac{n\kappa^2_1}{8C^2L^2\|\eta-\eta_0\|^2_{\infty}} \right \rbrace,\label{eq:hoeff1}
\end{align}
where $L>0$ is the Lipschitz constant associated with $H$. 
Here it is important to note that for $\eta\in\mathcal G_n$, $H(\eta)$ is Lipschitz in $\eta$ thanks to continuous differentiability of $H$, and boundedness of $\eta$ and $\eta'$
by the same constant on $\mathcal G_n$.
Also note that (\ref{eq:hoeff1}) holds irrespective of $x_i$; $i=1,\ldots,n$ being random or non-random
(see \ctn{Chatterjee18a}). 

Similarly, for $\eta\in\mathcal G_n$,
\begin{align}
& P\left(\left|  \dfrac{1}{n}  \sum_{i=1}^{n} \left\lbrace(1-y_i)\log\left(\dfrac{ 1- p(x_i)}{1-p_{0}(x_i)}\right)-E_{\mathbf{X}}\left((1-p_{0}(\mathbf{X})) \log \left\lbrace\dfrac{ \left(1- p(\mathbf{X})\right)}{\left(1- p_{0}(\mathbf{X})\right)} \right\rbrace\right) \right \rbrace  \right|> \dfrac{\kappa_1}{2} \right) \notag\\
	&\leq 4\exp\left\lbrace  -\dfrac{n\kappa^2_1}{8C^2L^2\|\eta-\eta_0\|^2_{\infty}} \right \rbrace.\label{eq:hoeff2}
\end{align}
Now,
\begin{align}
	& \sum_{n=1}^{\infty}\int_{s^{c}} P\left(\left| \dfrac{1}{n}  \sum_{i=1}^{n} \left\lbrace y_i\log\left(\dfrac{ p(x_i)}{p_{0}(x_i)}\right) - E_{\mathbf{X}} \left( p_{0}(\mathbf{X})\log \left\lbrace \dfrac{p(\mathbf{X})}{p_{0}(\mathbf{X})}\right\rbrace \right)\right\rbrace \right|> \dfrac{\kappa_1}{2} \right) \ d \pi(p)\notag\\
	&\leq\sum_{n=1}^{\infty}\int_{\mathcal G_n}4\exp\left\lbrace  -\dfrac{n\kappa^2_1}{8C^2L^2\|\eta-\eta_0\|^2_{\infty}} \right \rbrace d\pi(\eta)
	+\sum_{n=1}^{\infty}\pi\left(\mathcal G^c_n\right),\label{eq:sum1}
\end{align}
and
\begin{align}
	& \sum_{n=1}^{\infty}\int_{s^{c}} P\left(\left| \dfrac{1}{n} \sum_{i=1}^{n} \left\lbrace (1-y_i)\log\left(\dfrac{1-p(x_i)}{1-p_{0}(x_i)}\right) - 
	E_{\mathbf{X}} \left( (1-p_{0}(\mathbf{X}))\log \left\lbrace \dfrac{1-p(\mathbf{X})}{1-p_{0}(\mathbf{X})}\right\rbrace \right)\right\rbrace \right|> \dfrac{\kappa_1}{2} \right) \ d \pi(p)\notag\\
	&\leq\sum_{n=1}^{\infty}\int_{\mathcal G_n}4\exp\left\lbrace  -\dfrac{n\kappa^2_1}{8C^2L^2\|\eta-\eta_0\|^2_{\infty}} \right \rbrace d\pi(\eta)
	+\sum_{n=1}^{\infty}\pi\left(\mathcal G^c_n\right).\label{eq:sum2}
\end{align}
Then proceeding in the same way as (S-2.25) -- (S-2.30) of \ctn{Chatterjee18a}, and noting that $\sum_{n=1}^{\infty}\pi\left(\mathcal G^c_n\right)<\infty$,
we obtain (\ref{eq:finite1}). 

Hence (S6) holds.

\begin{rmk}
	\label{rmk:hoeff_binary}
	It is important to clarify the role of Assumption \ref{AA6} here. Note that, we need a lower bound for $\log\left(\frac{p(x)}{p_0(x)}\right)$.
	For instance, if $H(\eta(x))=\frac{\exp\left(\eta(x)\right)}{1+\exp\left(\eta(x)\right)}$,
	then even if $\|\eta\|_{\infty}\leq\sqrt{\beta n}$ on $\mathcal G_n$, it holds that $\log\left(\frac{p(x)}{p_0(x)}\right)\geq C-\sqrt{\beta n}$
	for all $x\in\mathfrak X$, for all $\eta\in\mathcal G_n$, for some constant $C$. In our bounding method uing the inequality $\log x\geq 1-1/x$ for $x>0$,
	we have $\log\left(\frac{p(x)}{p_0(x)}\right)\geq -\frac{\|p-p_0\|_{\infty}}{p(x)}\geq -2\exp\left(\sqrt{\beta n}\right)\|p-p_0\|_{\infty}$.
	It would then follow that the exponent of the Hoeffding inequality
	is $O(1)$. This would fail to ensure summability of the corresponding terms involving $V_i$. Thus, we need to ensure that $p(x)$ is bounded away
	from $0$. Similarly, the infinite sum associated with $W_i$ would not 
	be finite unless $1-p(x)$ is bounded away from $0$.
\end{rmk}


\subsection{Verification of (S7)for Binary Regression }
\label{subsec:S7}
This verification  follows from the fact that $h(p)$ is continuous. Indeed, 
for any set $A$ with  $\pi(A)> 0$, $\mathcal{G}_n \cap A \uparrow A$. It follows from continuity of $h$ that
$h\left(\mathcal{G}_n \cap A\right) \downarrow  h (A)$ as $n \rightarrow \infty$ and hence (S7) holds.



\newpage
\section{ Verification of (S1) to (S7) for Poisson regression}\label{PShaverify}

\subsection{Verification of (S1) for Poisson regression}
\label{subsec:PS1}

Observe that
\begin{align*}
& f_{\lambda}(\mathbf{Y_n}|\bX_n)=\prod_{i=1}^{n} f(y_i|x_i) = \prod_{i=1}^{n} \exp\left(-\lambda(x_i)\right)\dfrac{(\lambda(x_i))^{y_i}}{y_{i}!} ,  \\
& f_{\lambda_0}(\mathbf{Y_n}|\bX_n)=\prod_{i=1}^{n} f_{0}(y_i|x_i) = \prod_{i=1}^{n} \exp\left(-\lambda_{o}(x_i)\right)\dfrac{(\lambda_{0}(x_i))^{y_i}}{y_{i}!}. 
\end{align*}

Therefore,
\begin{equation}\label{PS11}
	R_{n}(\lambda)= \exp\left(-\sum_{i=1}^{n} [\lambda(x_i)-\lambda_{0}(x_i)]\right)\prod_{i=1}^{n}\left(\dfrac{\lambda(x_i)}{\lambda_{0}(x_i)}\right)^{y_i}
\end{equation}

and,
\begin{equation}
\frac{1}{n}\log R_{n}(\lambda)= \left(-\frac{1}{n}\sum_{i=1}^{n} [\lambda(x_i)-\lambda_{0}(x_i)]\right)+\frac{1}{n}\sum_{i=1}^{n}{y_i}\log\left(\dfrac{\lambda(x_i)}{\lambda_{0}(x_i)}\right).
	\label{eq:logR}
\end{equation}

Note that for any $a\in\Re$, $\left\{(y_i,\eta):y_i\log\left(\dfrac{\lambda(x_i)}{\lambda_{0}(x_i)}\right)<a\right\}
=\bigcup_{r=1}^{\infty}\left\{\eta:r\log\left(\dfrac{\lambda(x_i)}{\lambda_{0}(x_i)}\right)<a\right\}$.
Let $\tilde\eta_j$; $j=1,2,\ldots$ be such that $\|\eta_j-\eta\|_{\infty}\rightarrow 0$, as $j\rightarrow\infty$. Then, letting $\tilde\lambda_j(x)=H(\tilde\eta_j(x))$,
for all $x\in\mathfrak X$, it follows, since $0<C_1\leq\lambda(x)\leq C_2<\infty$ on $\mathfrak X$, that there exists $j_0\geq 1$ such that for $j\geq j_0$, 
$0<C_1\leq\tilde\lambda_j(x)\leq C_2<\infty$.
Hence, using the inequalities $1-\frac{1}{x}\leq\log x\leq x-1$ for $x>0$, we obtain $\left|\log\left(\frac{\tilde\lambda_j(x_i)}{\lambda(x_i)}\right)\right|
\leq C\|\tilde\lambda_j-\lambda\|_{\infty}$, for some $C>0$, for $j\geq j_0\geq 1$.
It follows that
\begin{align*}
	\left|r\log\left(\dfrac{\tilde\lambda_j(x_i)}{\lambda_{0}(x_i)}\right)-r\left(\dfrac{\lambda(x_i)}{\tilde\lambda_{0}(x_i)}\right)\right|
	=r\left|\log\left(\dfrac{\tilde\lambda_j(x_i)}{\lambda(x_i)}\right)\right|
	\leq  rC\|\tilde\lambda_j-\lambda\|_{\infty}
	\rightarrow 0,
\end{align*}
in the same way as in the binary regression, using Taylor's series expansion up to the first order.
Hence, $r\log\left(\dfrac{\lambda(x_i)}{\lambda_{0}(x_i)}\right)$ is continuous in $\eta$, ensuring measurability of
$\left\{\eta:r\log\left(\dfrac{\lambda(x_i)}{\lambda_{0}(x_i)}\right)<a\right\}$, and hence of 
$\left\{(y_i,\eta):y_i\log\left(\dfrac{\lambda(x_i)}{\lambda_{0}(x_i)}\right)<a\right\}$.
It follows that $\frac{1}{n}\sum_{i=1}^{n}{y_i}\log\left(\dfrac{\lambda(x_i)}{\lambda_{0}(x_i)}\right)$ is measurable.

Also, continuity of $\lambda(x_i)-\lambda_{0}(x_i)$ with respect to $\eta$ ensures measurability of $-\frac{1}{n}\sum_{i=1}^{n} [\lambda(x_i)-\lambda_{0}(x_i)]$.
Thus, $\frac{1}{n}\log R_{n}(\lambda)$, and hence $R_n(\lambda)$, is measurable.

\subsection{Verification of (S2) for Poisson regression }
\label{subsec:PS2}
For every $\lambda \in \Lambda$, we need to show that the KL divergence rate

\begin{equation*}
h(\lambda)=\underset{n\rightarrow\infty}{\lim}~\frac{1}{n}E_{\lambda _0}\left[\log\left\{\frac{f_{\lambda _0}(\bY_n|\bX_n)}{f_{\lambda }(\bY_n|\bX_n)}\right\}\right] =\underset{n\rightarrow\infty}{\lim}~\frac{1}{n}E_{\lambda _0}\left[-\log\left\{R_{n}(\lambda )\right\}\right].
\end{equation*}
exists (possibly being infinite) and is $\mathcal T$-measurable.

Now,
\begin{equation*}
\frac{1}{n}\log R_{n}(\lambda)= \left(-\frac{1}{n}\sum_{i=1}^{n} [\lambda(x_i)-\lambda_{0}(x_i)]\right)+\frac{1}{n}\sum_{i=1}^{n}{y_i} \log \left(\dfrac{\lambda(x_i)}{\lambda_{0}(x_i)}\right)
\end{equation*}
Therefore,

\begin{align*}\frac{1}{n}E_{\lambda_0}\left[-\log\left\{R_{n}(\lambda)\right\}\right]=  \left(\frac{1}{n}\sum_{i=1}^{n} [\lambda(x_i)-\lambda_{0}(x_i)]\right)+\frac{1}{n}\sum_{i=1}^{n}{\lambda_{0}(x_i)} \log \left(\dfrac{\lambda_{0}(x_i)}{\lambda(x_i)}\right).
\end{align*}
\begin{align}\underset{n\rightarrow\infty}{\lim}~\frac{1}{n}E_{\lambda_0}\left[-\log\left\{R_{n}(\lambda)\right\}\right]= & \underset{n\rightarrow\infty}{\lim}~\left(\frac{1}{n}\sum_{i=1}^{n} [\lambda(x_i)-\lambda_{0}(x_i)]\right)+ \underset{n\rightarrow\infty}{\lim}~\frac{1}{n}\sum_{i=1}^{n}{\lambda_{0}(x_i)} \log \left(\dfrac{\lambda_{0}(x_i)}{\lambda(x_i)}\right)\notag\\
= & E_{\mathbf{X}}\left[\lambda(\mathbf{X})-\lambda_{0}(\mathbf{X})\right]+ E_{\mathbf{X}}\left[{\lambda_{0}(\mathbf{X})} \log \left(\dfrac{\lambda_{0}(\mathbf{X})}{\lambda(\mathbf{X})}\right)\right]. \notag
\end{align}

The last line holds due to Assumption \ref{AA4} and SLLN. Here $E_{\mathbf{X}} (\cdot)=\int_{\mathfrak{X}} \cdot \ dQ$.
In other words,
\begin{equation}
h(\lambda)=E_{\mathbf{X}}\left[\lambda(\mathbf{X})-\lambda_{0}(\mathbf{X})\right]+ E_{\mathbf{X}}\left[{\lambda_{0}(\mathbf{X})} \log \left(\dfrac{\lambda_{0}(\mathbf{X})}{\lambda(\mathbf{X})}\right)\right]. \label{Ph_p}
\end{equation}

\subsection{Verification of (S3) for Poisson regression }
\label{subsec:PS3}
Here we need to verify the asymptotic equipartition property, that is, almost surely with respect to the true model $P$,
\begin{equation}
\underset{n\rightarrow\infty}{\lim}~\frac{1}{n}\log \left[R_n(\lambda)\right]=-h(\lambda)=\underset{n\rightarrow\infty}{\lim}~
	\frac{1}{n}E\left[\log\left\{\frac{f_{\lambda}(\bY_n|\bX_n)}{f_{\lambda0}(\bY_n|\bX_n)}\right\}\right].
\label{Peq:equipartition22}
\end{equation}

%
Now,
\begin{equation*}
-\frac{1}{n}\log R_{n}(\lambda)= \frac{1}{n}\sum_{i=1}^{n}\left\lbrace  \left [\lambda(x_i)-\lambda_{0}(x_i)\right]+y_i \log \left(\dfrac{\lambda_{0}(x_i)}{\lambda(x_i)}\right)\right\rbrace.
\end{equation*}
As before, for given $\lambda$, there exists $C>0$ such that $\left|\log\left(\frac{\lambda_0(x_i)}{\lambda(x_i)}\right)\right|\leq C\|\lambda-\lambda_0\|_{\infty}$. Hence,
\begin{align}
& \displaystyle \sum_{i=1}^{\infty} i^{-2} Var \left[ \left\lbrace  \left [\lambda(x_i)-\lambda_{0}(x_i)\right]+y_i \log \left(\dfrac{\lambda_{0}(x_i)}{\lambda(x_i)}\right)\right\rbrace \right] \notag\\
&= \displaystyle \sum_{i=1}^{\infty} i^{-2}\lambda_{0}(x_i) \left[ \log \left(\dfrac{\lambda_{0}(x_i)}{\lambda(x_i)}\right)\right]^{2}\notag\\
	&\leq C^2\|H(\kappa_0)\|\left(\|\lambda-\lambda_0\|_{\infty}\right)^2
	\displaystyle\sum_{i=1}^{\infty}i^{-2}\notag\\
	&<\infty.
\end{align}
Observe that $y_i$ are  observations from independent random variables. 
Hence from Kolmogorov’s SLLN for independent random variables and from Assumption \ref{AA4}, \eqref{Peq:equipartition22} holds as $n \rightarrow \infty$.

\subsection{Verification of (S4) for Poisson regression }
\label{subsec:PS4}

If $I=\{\lambda: \ h(\lambda)=\infty\}$ then we need to show $\Pi(I)<1$. 
But this holds in almost the same way as for binary regression.
In other words, (S4) holds for Poisson regression.

\subsection{Verification of (S5) for Poisson regression }
\label{subsec:PS5}

The parameter space here remains the same as in the binary regression case, that is, $\Theta=\mathcal{C'}(\mathfrak{X})$. 
We also consider the same sequence $\mathcal G_n$ as in binary regression.
We need to verify that
\begin{enumerate}
	\item $h\left(\mathcal G_n\right)\rightarrow h\left(\Lambda\right)$, as $n\rightarrow\infty$;
	\item The inequality $\pi\left(\mathcal G_n\right)\geq 1-\alpha\exp\left(-\beta n\right)$ holds for some $\alpha>0, \beta> 2h(\Lambda)$; 
	
	\item The convergence in (S3) is uniform over $\mathcal G_n\setminus I$.
\end{enumerate}






\subsubsection{ Verification of (S5) (1)}
\label {subsubsec:PS51}
We now need to verify that $h\left(\mathcal G_n\right)\rightarrow h\left(\Lambda\right)$ as $n\rightarrow\infty$. 
But this holds in the same way as for binary regression.

\subsubsection{ Verification of (S5) (2)}
\label {subsubsec:PS52}
Again, this holds in the same way as for binary regression.
%
%


\subsubsection{ Verification of (S5) (3)}
\label {subsubsec:PS53}

Using the same arguments as in the binary regression case, here we only need
to show that $\dfrac{1}{n} \log( R_n(\lambda)) $ and $h(\lambda)$ are both Lipschitz. 
 
Recall that
\begin{align*} \frac{1}{n}  \log R_{n}(\lambda)= \frac{1}{n}\sum_{i=1}^{n}\left\lbrace  \left [\lambda_{0}(x_i)-\lambda(x_i)\right]+y_i \log \left(\dfrac{\lambda(x_i)}{\lambda_{0}(x_i)}\right)\right\rbrace.
\end{align*}
For any $\eta_1,\eta_2\in\mathcal G$, there exists $C>0$ such that $\left|\log\left(\frac{\lambda_1(x)}{\lambda_2(x)}\right)\right|\leq C\|\lambda_1-\lambda_2\|_{\infty}$,
for all $x\in\mathfrak X$, where $\lambda_1=H(\eta_1)$ and $\lambda_2=H(\eta_2)$. Hence,
\begin{align*}
&\left|\frac{1}{n} \log R_{n}(\lambda_1)-\frac{1}{n}  \log R_{n}(\lambda_2)\right|
\leq\|\lambda_1-\lambda_2\|_{\infty}\left(1+C\times\frac{1}{n}\sum_{i=1}^ny_i\right).
\end{align*}
Thus, $\frac{1}{n}\log R_{n}(\lambda)$ is almost surely Lipschitz with respect to $\lambda$.
Since, by Kolmogorov's SLLN for independent variables, 
$\frac{1}{n}\sum_{i=1}^ny_i\stackrel{a.s}{\longrightarrow} E_{\bX}\left(\lambda_0(\bX)\right)<\infty$, as $n\rightarrow\infty$, 
and since $\lambda=H(\eta)$ is Lipschitz in $\eta\in\mathcal G_n$ in the same way as in binary regression, the desired stochastic equicontinuity follows. 
Lipschitz continuity of $h(\lambda)$ in $\mathcal G_n$ follows using similar techniques.

\subsection{Verification of (S6)  for Poisson Regression}
\label{subsec:PS6}
Since
\begin{align}
&\sum_{n=1}^{\infty}\int_{S^c} P\left(\left| \dfrac{1}{n} \log  R_{n}(\lambda) + h(\lambda) \right|> \kappa -h(\Lambda) \right) \ d\pi(\lambda)\notag\\ 
&\qquad\leq\sum_{n=1}^{\infty}\int_{\mathcal G_n} P\left(\left| \dfrac{1}{n} \log  R_{n}(\lambda) + h(\lambda) \right|> \kappa -h(\Lambda) \right) \ d\pi(\lambda)\notag\\
	&\qquad\quad+\sum_{n=1}^{\infty}\int_{\mathcal G^c_n} P\left(\left| \dfrac{1}{n} \log  R_{n}(\lambda) + h(\lambda) \right|> \kappa -h(\Lambda) \right) \ d\pi(\lambda)\notag\\
&\qquad\leq\sum_{n=1}^{\infty}\int_{\mathcal G_n} P\left(\left| \dfrac{1}{n} \log  R_{n}(\lambda) + h(\lambda) \right|> \kappa -h(\Lambda) \right) \ d\pi(\lambda)
	+\sum_{n=1}^{\infty}\pi\left(\mathcal G^c_n\right),\label{eq:finite2}
\end{align}
and the second term of (\ref{eq:finite2}) is finite, it is enough to show that the first term of (\ref{eq:finite2}) is finite.

Let us take $\kappa_1=\kappa -h(\Lambda)$.
Observe that for $\eta\in\mathcal G_n$, 
\begin{align}
&P\left(\left| \dfrac{1}{n} \log  R_{n}(\lambda) + h(\lambda) \right|> \kappa_1 \right)\notag\\
	&\leq P\left( \left|\frac{1}{n}\sum_{i=1}^n\left[\lambda_0(x_i)\log\left(\frac{\lambda(x_i)}{\lambda_0(x_i)}\right)-
	E_{\bX}\left(\lambda_0(\bX)\log\left(\frac{\lambda(\bX)}{\lambda_0(\bX)}\right)\right)\right]\right|>\frac{\kappa_1}{3}\right)\label{eq:term1}\\
	&\qquad+P\left(\left|\frac{1}{n}\sum_{i=1}^n\left[\left(\lambda_0(x_i)-\lambda(x_i)\right)-E_{\bX}\left(\lambda_0(\bX)-\lambda(\bX)\right)\right]\right|
	>\frac{\kappa_1}{3}\right)\label{eq:term2}\\
	&\qquad+P\left(\left|\frac{1}{n}\sum_{i=1}^n\left[y_i\log\left(\frac{\lambda(x_i)}{\lambda_0(x_i)}\right)
	-\lambda_0(x_i)\log\left(\frac{\lambda(x_i)}{\lambda_0(x_i)}\right)\right]\right|>\frac{\kappa_1}{3}\right).\label{eq:term3}
\end{align}
Using Hoeffding's inequality and Lipschitz continuity of $H$ in $\mathcal G_n$ as in binary regression, we find that  
(\ref{eq:term1}) and (\ref{eq:term2}) are bounded above by 
$2\exp\left(-\frac{C_1n\kappa^2_1}{\|\eta-\eta_0\|^2_{\infty}}\right)$, and $\exp\left(-\frac{C_2n\kappa^2_1}{\|\eta-\eta_0\|^2_{\infty}}\right)$,
for some $C_1>0$ and $C_2>0$. These bounds hold even if the covariates are non-random.

To bound (\ref{eq:term3}), we shall first show that the summands are sub-exponential, and then shall apply Bernstein's inequality (see, for example,
\ctn{Uspensky37}, \ctn{Bennett62}, \ctn{Massart03}). 
Direct calculation yields
\begin{align}
&E\left[\exp\left\{t\left(y_i\log\left(\frac{\lambda(x_i)}{\lambda_0(x_i)}\right)-\lambda_0(x_i)\log\left(\frac{\lambda(x_i)}{\lambda_0(x_i)}\right)\right)\right\}\right]\notag\\
	&=\exp\left[-t\lambda_0(x_i)\log\left(\frac{\lambda(x_i)}{\lambda_0(x_i)}\right)\right]\times
	\exp\left[\lambda_0(x_i)\left\{\exp\left(t\log\left(\frac{\lambda(x_i)}{\lambda_0(x_i)}\right)\right)-1\right\}\right].
	\label{eq:subexp1}
\end{align}
The first factor of (\ref{eq:subexp1}) has the following upper bound:
\begin{equation}
	\exp\left[-t\lambda_0(x_i)\log\left(\frac{\lambda(x_i)}{\lambda_0(x_i)}\right)\right]
	\leq \exp\left(c_{\lambda}\|\lambda\|_{\infty}|t|\right).
	\label{eq:subexp2}
\end{equation}
A bound for the second factor of (\ref{eq:subexp1}) is given as follows: 
\begin{align}
&\exp\left[\lambda_0(x_i)\left\{\exp\left(t\log\left(\frac{\lambda(x_i)}{\lambda_0(x_i)}\right)\right)-1\right\}\right]\notag\\
&\qquad	\leq\exp\left[\|\lambda_0\|_{\infty}\left(\exp\left(\frac{t\|\lambda-\lambda_0\|_{\infty}}{\kappa_P}\right)-1\right)\right]\notag\\
	&\qquad\leq\exp\left[\|\lambda_0\|_{\infty}\left(c_{\lambda}|t|+c^2_{\lambda}t^2\right)\right],
	\label{eq:subexp3}
\end{align}
for $|t|\leq c^{-1}_{\lambda}$, where $c_{\lambda}=C\|\lambda-\lambda_0\|_{\infty}$, for some $C>0$.

Combining (\ref{eq:subexp1}), (\ref{eq:subexp2}) and (\ref{eq:subexp3}) we see that (\ref{eq:subexp1}) is bounded above by $\exp\left(c^2_{\lambda}t^2\right)$
provided that 
\begin{equation}
c_{\lambda}|t|\geq 2/\left(\|\lambda_0\|^{-1}_{\infty}-1\right)\geq 2/\left(\kappa^{-1}_P-1\right).
\label{eq:subexp4}
\end{equation}
The rightmost bound of (\ref{eq:subexp4}) is close to zero if $\kappa_P$ is chosen sufficiently small.
Now consider the function $g(t)=\exp\left(c^2_{\lambda}t^2\right)-f(t)$, where $f(t)$ is given by (\ref{eq:subexp1}). Since $g(t)$ is continuous in $t$ and
$g(0)=0$ and $g(t)>0$ on $2/\left(\kappa^{-1}_P-1\right)\leq |t|\leq c^{-1}_{\lambda}$, it follows that on the sufficiently small interval 
$0\leq |t|\leq 2/\left(\kappa^{-1}_P-1\right)$, $g(t)>0$. In other words, (\ref{eq:subexp1}) is bounded above by $\exp\left(c^2_{\lambda}t^2\right)$
for $0\leq |t|\leq c^{-1}_{\lambda}$. Thus, 
$z_i=y_i\log\left(\frac{\lambda(x_i)}{\lambda_0(x_i)}\right)-\lambda_0(x_i)\log\left(\frac{\lambda(x_i)}{\lambda_0(x_i)}\right)$ are independent
sub-exponential variables with parameter $c_{\lambda}$.

Bernstein's inequality, in conjunction with Lipschitz continuity of $H$ on $\mathcal G_n$ then ensures that (\ref{eq:term3}) is bounded above by 
$2\exp\left[-\frac{n}{2}\min\left\{\frac{C_1\kappa^2_1}{\|\eta-\eta_0\|^2_{\infty}},\frac{C_2\kappa_1}{\|\eta-\eta_0\|_{\infty}}\right\}\right]$, for positive
constants $C_1$ and $C_2$.

The rest of the proof of finiteness of (\ref{eq:finite2}) follows in the same (indeed, simpler) way as \ctn{Chatterjee18a}.
Hence (S6) holds.


\begin{rmk}
\label{rmk:hoeff_poisson}
Arguments similar to that of Remark \ref{rmk:hoeff_binary} shows that it is essential to have $\lambda$ bounded away from zero. 
\end{rmk}

\subsection{Verification of (S7) for Poisson regression}
\label{subsec:PS7}
This verification follows from the fact that $h(\lambda)$ is continuous, similar to binary regression.


\bibliography{irmcmc}

\end{document}